\documentclass[a4paper, english,11pt, reqno]{amsart}

\newcommand{\cE}{\mathcal{E}}
\newcommand{\cF}{\mathcal{F}}

\newcommand{\cH}{\mathcal{H}}
\newcommand{\cI}{\mathcal{I}}

\newcommand{\cM}{\mathcal{M}}

\newcommand{\cO}{\mathcal{O}}

\newcommand{\cX}{\mathcal{X}}

\usepackage{amsmath, amsfonts, dsfont,amssymb, amsthm, tikz-cd, enumerate, hyperref,mathtools, mathrsfs, geometry, verbatim, lilyglyphs,marginnote, romannum, appendix, todonotes}
\usepackage{graphicx}

\DeclareRobustCommand{\SkipTocEntry}[5]{}

\newcommand{\ddc}{\mathrm{dd^c}}

\newcommand{\ia}{\mathfrak{a}}

\newcommand{\C}{\mathds C}
\newcommand{\ind}{\mathds 1}

\newcommand{\ideal}{\mathscr I}

\newcommand{\Q}{\mathds Q}

\newcommand{\R}{\mathds R}

\newcommand{\Pro}{\mathds P^1}

\DeclareMathOperator{\PL}{PL}

\DeclareMathOperator{\MA}{MA}

\DeclareMathOperator{\tspec}{TropSpec}
\DeclareMathOperator{\trivial}{triv}

\DeclareMathOperator{\supp}{supp}
\DeclareMathOperator{\Aff}{Aff}

\DeclareMathOperator{\VCar}{VCar}

\DeclareMathOperator{\chern}{c}
\DeclareMathOperator{\NA}{{na}}
\DeclareMathOperator{\PSH}{PSH}
\DeclareMathOperator{\CPSH}{CPSH}

\DeclareMathOperator{\divisorial}{{div}}

\DeclareMathOperator{\ord}{ord}

\DeclareMathOperator{\vol}{vol}

\newcommand{\Xdiv}{X^{\divisorial}}

\makeatletter
\newcommand{\newreptheorem}[2]{\newtheorem*{rep@#1}{\rep@title}\newenvironment{rep#1}[1]{\def\rep@title{#2 \ref*{##1}}\begin{rep@#1}}{\end{rep@#1}}}
\makeatother

\newreptheorem{theorem}{Theorem}
\newreptheorem{lemma}{Lemma}

\newtheorem{defi}{Definition}[subsection]
\newtheorem{remark}[defi]{Remark}
\newtheorem{example}[defi]{Example}
\newtheorem*{example*}{Example}
\newtheorem{lemma}[defi]{Lemma}
\newtheorem{prop}[defi]{Proposition}
\newtheorem{corollary}[defi]{Corollary}
\newtheorem{theorem}[defi]{Theorem}
\newtheorem{theorema}{Theorem}

\numberwithin{equation}{subsection}

\AtBeginDocument{\pagenumbering{arabic}}

\newcommand{\Md}{{\cM^{\divisorial}}}

\newcommand{\DcX}{\Delta_\cX}

\newcommand{\E}{\mathrm E}
\newcommand{\T}{\mathrm T}
\newcommand{\Ev}{\mathrm E^\vee}
\newcommand{\M}{\mathrm M}
\newcommand{\DF}{\mathrm{DF}}
\newcommand{\He}{\mathrm H}
\newcommand{\Jd}{\mathrm J}
\newcommand{\Pe}{\mathrm P}
\DeclareMathOperator{\np}{np}
\DeclareMathOperator{\psef}{psef}
\DeclareMathOperator{\Ent}{Ent}
\newcommand{\Xnp}{X^{\np}}

\newcommand{\dtz}{\frac{\mathrm d}{\mathrm dt}\Bigr\rvert_{t = 0}}

\hypersetup{
    colorlinks=true,
    linkcolor=blue,
    filecolor=magenta,      
    urlcolor=cyan,
    pdftitle={A transcendental non-Archimedean Calabi-Yau Theorem},
}
\title[A transcendental non-Archimedean Calabi--Yau Theorem]{A transcendental non-Archimedean Calabi--Yau Theorem with applications to the cscK problem}
\author[Pietro Mesquita Piccione \and David Witt Nyström]{Pietro Mesquita-Piccione \and David Witt Nyström}
\begin{document}
\maketitle

\begin{abstract}
Let $X$ be a compact K\"ahler manifold and $\alpha$ a K\"ahler class on $X$. We prove that if $(X,\alpha)$ is uniformly K-stable for models, then there is a unique cscK metric in $\alpha$. This was first proved in the algebraic case by Chi Li \cite{Li22geodesic,Li23fujita}, and it strengthens a related result in \cite{MP24transcendental}. K-stability for models is defined in terms of big test configurations, but we also give a valuative criterion as in \cite{BJ23nakstabii} together with an explicit formula for the associated $\beta$-invariant. To accomplish this we further develop the non-Archimedean pluripotential theory in the transcendental setting, as initiated in \cite{DXZ23transcendental} and \cite{MP24transcendental}. In particular we prove the continuity of envelopes and orthogonality properties, and using that, we are able to extend the non-Archimedean Calabi-Yau Theorem in \cite{BJ22trivval} to the general K\"ahler setting.
\end{abstract}

\setcounter{tocdepth}{1}
\tableofcontents

\section{Introduction}

\addtocontents{toc}{\SkipTocEntry}
\subsection{The cscK problem and the YTD conjecture}

Let $(X,\omega)$ be a compact Kähler manifold and let $\alpha:=\{\omega\}\in H^{1,1}(X,\R)$ be the associated K\"ahler class. A major open problem is to decide when $\alpha$ contains a cscK metric, i.e. a K\"ahler metric with constant scalar curvature (for simplicity we do not distinguish between Kähler forms and metrics). A possible answer is given by the Yau-Tian-Donaldson (YTD) conjecture which in particular says that there should be a unique cscK metric in $\alpha$ if and only if $(X,\alpha)$ is uniformly K-stable. The conjecture was famously proved by Chen-Donaldson-Sun \cite{CDS15kahler1, CDS15kahler2, CDS15kahler3} in the important case where $X$ is Fano and $\alpha=-c_1(K_X)$, but the general case remains open. 

Recall that K-stability is defined in terms of certain degenerations of $(X,\alpha)$ known as test configurations.

\begin{defi} A (smooth dominating) test configuration $(\mathcal{X},A+D)$ of $(X,\alpha)$ consists of the following data:
\begin{enumerate}
    \item a compact K\"ahler manifold $\mathcal{X}$ together with a surjective map $\pi:\mathcal{X}\to X\times \mathbb{P}^1$ such that $\pi: \mathcal{X}\setminus \mathcal{X}_0\to X\times (\mathbb{P}^1\setminus \{0\})$ is a biholomorphism, $\mathcal{X}_0$ denoting the zero fiber,
    \item a lift of the standard $\mathbb{C}^*$-action on $X\times \mathbb{P}^1$ to $\mathcal{X}$ making $\pi$ equivariant,
    \item a class $A+D\in H^{1,1}(\mathcal{X},\R)$ where $A:=\pi_X^*(\alpha)$ and $D$ is a vertical divisor, i.e. a divisor supported on $\mathcal{X}_0$ (for convenience we do not distinguish between $D$ and its cohomology class).
\end{enumerate}

\end{defi}

We say that the test configuration $(\mathcal{X},A+D)$ is K\"ahler/big if $A+D$ is K\"ahler/big (recall that a class is said to be \emph{pseudoeffective} if it contains a closed positive current, and \emph{big} if it can be written as a sum of a K\"ahler class and a pseudoeffective class). We also call $\mathcal{X}$ a model of $X$. If $\mathcal{X}$ is a model and $Y\subseteq \mathcal{X}_0$ is a $\mathbb{C}^*$-invariant submanifold the blow-up of $Y$ in $\mathcal{X}$ is a new model. A simple way to produce models is thus to start from the trivial model $X\times \mathbb{P}^1$ and then to iterate this blow-up procedure. 

Many different invariants of a K\"ahler/big test configuration can now be defined in terms of intersection numbers, including the Donaldson-Futaki invariant 
$$\DF(\mathcal{X},A+D):=K_{\mathcal{X}/\mathbb{P}^1}\cdot\langle(A+D)^n\rangle -\frac{n\alpha^{n-1}\cdot K_X}{(n+1)\alpha^n}\langle(A+D)^{n+1}\rangle,$$ the closely related Mabuchi invariant $$\M^{\NA}(\cX, A+D):=\DF(\mathcal{X},A+D)-(\cX_0-\cX_0^{red})\cdot \langle(A+D)^n\rangle$$ and the $J$-invariant $$\Jd^{\NA}(\mathcal{X},A+D):=\langle A+D\rangle\cdot A^n-\frac{1}{n+1}\langle(A+D)^{n+1}\rangle.$$
Here $\langle(A+D)^k\rangle$ denotes the positive (or movable) intersection classes introduced in \cite{Bou02cones}, which are equal to $(A+D)^k$ when $A+D$ is K\"ahler (see Section \ref{sec:posint} for details).

\begin{defi}\label{def:models}
    We say that $(X,\alpha)$ is uniformly K-stable if there is some $\delta>0$ such that for all K\"ahler test configurations $(\mathcal{X},A+D)$ we have that $$\M^{\NA}(\mathcal{X},A+D)\geq \delta \Jd^{\NA}(\mathcal{X},A+D).$$ If the same is true for all big test configurations we say that $(X,\alpha)$ is uniformly K-stable for models. 
\end{defi}

\begin{remark}
    K-stability was originally only defined for algebraic $(X,\alpha)$, i.e. when $X$ is projective and $\alpha=c_1(L)$ for some ample ($\R$-)line bundle $L$. The general case (i.e. allowing $X$ to be non-projective and/or $\alpha$ to be non-rational) was first considered in \cite{SD18} and \cite{DR17kstability}, where they independently proved that the existence of a cscK metric in $\alpha$ implies that $(X,\alpha)$ is $K$-stable. K-stability for models was introduced by Chi Li \cite{Li22geodesic, Li23fujita}.
\end{remark}

Even though the general YTD-conjecture remains open, there has been some impressive progress since the solution of the Fano case. E.g., in \cite{Li22geodesic, Li23fujita} Chi Li proved that for algebraic $(X,\alpha)$, uniform K-stability for models implies the existence of a unique cscK metric in $\alpha$. Our first main result extends Chi Li's result to the general K\"ahler case:

\begin{theorema} \label{thm:A} 
If $(X,\alpha)$ is uniformly K-stable for models, then there is a unique cscK metric in $\alpha$.
\end{theorema}

In addition we give a \emph{valuative criterion} for $K$-stability for models,  see Section~\ref{sec:valuative}. More precisely we define a $\beta$-invariant on the set of divisorial measures $\mu\in \Md$, and we prove that $(X,\alpha)$ is $K$-stable for models iff there is a $\delta>0$ such that for all $\mu\in \Md$, $\beta_A(\mu)\ge \delta \Ev_A(\mu)$, where $\Ev_A(\mu)$ denotes the energy of $\mu$. This extends results of Boucksom--Jonsson in the algebraic case \cite{BJ23nakstabii}, which in turn was inpired by the work of Fujita \cite{Fuj19valuative} and Li \cite{Li17Ksemi} in the Fano case. In Section~\ref{sec:computingbeta} we also give an explicit formula for $\beta_A(\mu_{\xi})$ where $\mu_{\xi}:=\sum_{i=1}^l\xi_i\delta_{\ord_{F_i}}$, the $F_i$:s being prime divisors on a modification $\mu: X'\to X$ of $X$. Namely, we get that $$\beta_A(\mu_{\xi})=\sum_i \xi_i A_X(F_i)+\nabla_{K_X}(\widehat{-f_A})(-\xi),$$ where $A_X(F_i)$ denotes the log discrepancy of $F_i$, $$f_A(t)=\min_i(t_i)+V^{-1}\int_{\min_i(t_i)}^{+\infty}\langle(\mu^*\alpha-\sum_{i=1}^l(\lambda-t_i)_+F_i)^n\rangle d\lambda,$$ and $\widehat{-f_A}$ denotes the Legendre transform of $-f_A$ (see Corollary \ref{cor:betaformula}). Our proof relies in part on some results from \cite{DXZ23transcendental}.

\addtocontents{toc}{\SkipTocEntry}
\subsection{Geodesic stability and non-Archimedean pluripotential theory}
Chi Li's proof of Theorem \ref{thm:A} in the algebraic case relied on some big technical and conceptual advances made in the last ten years or so.

Developments in (Archimedean) pluripotential theory has demonstrated the benefits of working with the space of finite energy potentials $\cE^1_{\omega}(X)$ (see e.g. \cite{BBGZ13variational,Dar17completion}). Notably Darvas and He showed in \cite{DH17geodesic}    how geodesic rays in $\cE^1_{\omega}(X)$ can be constructed to detect instability.

In \cite{CC21csck1, CC21csck2} Chen-Cheng then proved some very strong estimates which allowed them to prove that if $(X,\alpha)$ is geodesically stable (meaning that the Mabuchi functional $M_{\omega}$ is coercive along any non-trivial finite energy geodesic ray), then there is a unique cscK metric in $\alpha$. 

Also underlying Chi Li's proof is the non-Archimedean pluripotential theory, which after being initiated by Kontsevich--Tschinkel \cite{KT01kahler} has been extensively developed by Boucksom, Jonsson, Favre and others (see e.g. \cite{BFJ15solution, BFJ16semipositive, BJ22trivval, BJ24green}).

A key observation in \cite{Li22geodesic} which ultimately allowed Chi Li to prove Theorem \ref{thm:A} in the algebraic case is that any destabilizing geodesic ray must be maximal, i.e. must correspond to a non-Archimedean potential (see e.g. \cite{BBJ21YTD}). This ties geodesic stability to a non-Archimedean stability notion called $\widehat{K}$-stability. 

The results of Chen-Cheng on geodesic stability are valid in the general K\"ahler setting. Non-Archimedean pluripotential theory on the other hand was originally only formulated for algebraic $(X,\alpha)$. Recently though, Darvas-Xia-Zhang \cite{DXZ23transcendental} and the first named author \cite{MP24transcendental} proposed two somewhat different ways of extending non-Archimedean pluripotential theory to the general setting. As an application the first named author generalized a result of Chi Li, closely related to but weaker than Theorem \ref{thm:A}, which says that uniform $\widehat{K}$-stability implies the existence of a unique cscK metric (see \cite{MP24transcendental}).

In this paper we further develop the non-Archimedean pluripotential theory for general $(X,\alpha)$, and as an application we are able to prove Theorem \ref{thm:A} in the general case.

\addtocontents{toc}{\SkipTocEntry}
\subsection{A non-Archimedean Calabi-Yau Theorem}

Key to proving Theorem \ref{thm:A} is our second main result, namely a non-Archimedean Calabi--Yau Theorem for general $(X,\alpha)$. 
A non-Archimedean Calabi--Yau Theorem for algebraic $(X,\alpha)$ was proved by Boucksom--Favre--Jonsson \cite{BFJ15solution}, subsequently followed by other versions (see e.g. \cite{BJ22trivval, BGM22differentiability}). 
Before describing these results though it makes sense to briefly discuss the original Calabi--Yau Theorem.

\subsubsection{The Calabi--Yau Theorem} \label{sec:introCY}

The celebrated Calabi--Yau Theorem proved by Yau \cite{Yau77calabi, Yau78ricci} says that for any volume form $dV$ on $X$ such that $\int_XdV=\int_X\omega^n=\alpha^n$ there is a unique K\"ahler form $\omega'$ in $\alpha$ such that $(\omega')^n=dV$. 

Other versions have since been established by e.g. Kolodziej \cite{Kol98complex} and Berman--Boucksom--Guedj--Zeriahi \cite{BBGZ13variational}. To formulate this last version, which is of particular relevance to our paper, we first need to remind the reader about some basic definitions in (Archimedean) pluripotential theory. 

A smooth function $\varphi$ on $X$ is called a K\"ahler potential (with respect to $\omega$), written $\varphi\in\mathcal{H}_{\omega}(X)$, if $\omega+dd^c\varphi$ is K\"ahler. A decreasing limit $\psi$ of K\"ahler potentials not identically equal to $-\infty$ is said to be $\omega$-psh, written $\psi\in PSH_{\omega}(X)$. 

The energy of a K\"ahler potential is defined as $$\E_{\omega}(\varphi):=\frac{V^{-1}}{n+1}\sum_{j=0}^n\int_X \varphi(\omega+dd^c\varphi)^j\wedge \omega^{n-j},$$ where $V:=\alpha^n$. The energy of $\psi\in PSH_{\omega}(X)$ is defined as the infimum of the energy of all K\"ahler potentials $\varphi\geq \psi$, and the space of finite energy potentials is defined as $$\mathcal{E}^1_{\omega}(X):=\{\psi\in PSH_{\omega}(X): \E_{\omega}(\psi)>-\infty\}.$$ We also let $\cE^1_{\omega,\sup}(X):=\{\varphi\in \mathcal{E}^1_{\omega} : \sup \varphi=0\}$.

The Monge-Amp\`ere measure of a K\"ahler potential $\varphi$ is defined as $$\MA_{\omega}(\varphi):=V^{-1}(\omega+dd^c\varphi)^n$$ and there is a natural extension of the Monge-Amp\`ere operator to $\mathcal{E}^1_{\omega}(X)$. 

The (dual) energy $\Ev_{\omega}(\mu)$ of a Radon probabiltiy measure is defined as $$\Ev_{\omega}(\mu):=\sup\left\{\E_{\omega}(\varphi)-\int_X\varphi\, \mathrm d\mu: \varphi\in \mathcal{E}^1_{\omega}(X)\right\},$$ those with $\Ev_{\omega}(\mu)<\infty$ giving us the space of finite energy measures $\mathcal{M}^1_{\omega}(X)$.

The Calabi--Yau Theorem proved in \cite{BBGZ13variational} is the following:

\begin{theorem} \label{thm:CYweak}
The Monge-Amp\`ere operator is a bijection between $\mathcal{E}^1_{\omega,\sup}(X)$ and \\ $\mathcal{M}^1_{\omega}(X)$. 
\end{theorem}

To explain the non-Archimedean analogue we need to recall some basic definitions of non-Archimedean pluripotential theory.

\subsubsection{Berkovich/tropical analytification}
In the algebraic case one would start with the Berkovich analytification $X^{an}$ of $X$ with respect to the trivial norm on $\mathbb{C}$ (see e.g. \cite{BJ22trivval}). However, if $X$ is non-projective $X^{an}$ might be trivial, so following \cite{MP24transcendental} we instead consider the tropical analytification $X^{\NA}$ (denoted by $X^{\beth}$ in \cite{MP24transcendental}) of $X$ whose points correspond to semivaluations on the set $\ideal_X$ of coherent ideal sheaves on $X$. Recall that a nonconstant function $v: \ideal_X\to [0,\infty]$ is called a semivaluation if for any $I,J\in\ideal_X$ we have that $v(IJ)=v(I)+v(J)$ and $v(I+J)=\min(v(I),v(J))$. Equipped with the topology of pointwise convergence $X^{\NA}$ becomes a compact Hausdorff space. When $X$ is projective there is a natural identification between $X^{an}$ and $X^{\NA}$. For details see \cite{MP24transcendental}.

\subsubsection{Divisorial points} If $E_i$ is an irreducible vertical divisor on a model $\mathcal{X}$, then there is an associated semivaluation $v_{E_i}(I):=\min\{b_i^{-1}\ord_{E_i}(f\circ \pi_X): f\in I(U), U\subseteq X\}$, where $b_i$ is the order of vanishing of $\cX_0$ along $E_i$. 
Such valuations are called \emph{divisorial valuations} and the set of divisorial points in $X^{\NA}$ is denoted by $\Xdiv$. Importantly $\Xdiv$ is dense in $X^{\NA}$ (see \cite[Theorem B]{MP24transcendental}). 

\subsubsection{Dual complexes}

To each SNC model $\cX$ (i.e. a model such that $(\mathcal{X},\mathcal{X}_0^{red})$ is a SNC pair) one can associate a \emph{dual complex} $\Delta_{\cX}$ in the following way. 
Let $I$ be the index set for the irreducible components $E_i$ of $\cX_0$. To each irreducible component $Z$ of an intersection $\bigcap_{i\in J} E_i$, where $J\subseteq I$, we associate the simplex $\Delta_Z:=\{w\in (\R_{\ge 0})^{|J|}\mid \sum_{i\in J} w_i b_i\le 1\},$ where as usual $b_i$ denotes the multiplicity of $\cX_0$ along $E_i$.
This collection then defines a simplicial complex that we denote by $\Delta_{\cX}$. Morphisms between SNC models induce simplicial maps on the associated dual complexes, and one can show that there is a natural identification between $X^{\NA}$ and the projective limit of $\Delta_{\cX}$. Using so-called monomial valuations there is also a natural injection $i_{\cX}: \Delta_{\cX}\hookrightarrow X^{\NA}$, hence we can think of $\Delta_{\cX}$ as a subset of $X^{\NA}$. For more details see \cite{MP24transcendental}.

\subsubsection{Vertical divisors and PL functions} Let $D$ be a (not necessarily irreducible) vertical divisor on a model $\mathcal{X}$. If $x\in \Xdiv$ we let $\mathcal{X}'$ be a SNC model that dominates $\mathcal{X}$ and with an irreducible vertical divisor $E_i$ such that $x=v_{E_i}$. We can then write $\mu^*(D)=a_iE_i+\sum_{j\neq i} a_jE_j$ where $\mu$ is the map from $\mathcal{X}'$ to $\mathcal{X}$ and $(E_j)_{j\neq i}$ are the other irreducible vertical divisors on $\mathcal{X}'$. Then $f_D(x):=b_i^{-1}a_i$ defines a function $f_D$ on $\Xdiv$ which can be seen to have a continuous extension to the whole of $X^{\NA}$, also denoted by $f_D$. Functions of this kind are called piecewise linear (PL), the set of PL functions being denoted by $\PL(X^{\NA})$.

\subsubsection{K\"ahler potentials and $A$-psh functions} A PL function $f_D=\varphi_D$ is said to be a K\"ahler potential (with respect to $A$), written $\varphi_D\in \mathcal{H}_A(X^{\NA})$, if $A+D$ is relatively Kähler on some model $\mathcal{X}$ (relatively Kähler here means that $A+D+c\mathcal{X}_0$ is Kähler for large $c$). A decreasing limit $\varphi$ of K\"ahler potentials $\varphi_{D_i}$ is said to be $A$-psh (or just psh), written $\varphi\in PSH_A(X^{\NA})$.

\subsubsection{Finite energy potentials} The energy $\E_A(\varphi_D)$ of a K\"ahler potential $\varphi_D$ is defined as $$\E_A(\varphi_D):=\frac{V^{-1}}{n+1}(A+D)^{n+1},$$ the energy $\E_A(\varphi)$ of an $A$-psh function $\varphi$ is defined as the infimum of the energy of all K\"ahler potentials $\varphi_D\geq \varphi$, and the space of finite energy potentials is defined as $$\mathcal{E}^1_A(X^{\NA}):=\{\varphi\in PSH_A(X^{\NA}): \E_A(\varphi)>-\infty\}.$$ We equip $\mathcal{E}^1_A(X^{\NA})$ with the strong topology, defined as the coursest topology, finer than the topology of pointwise convergence on $\Xdiv$, such that $\E_A$ becomes continuous. We also let $\cE^1_{A,\sup}(X^{\NA}):=\{\varphi\in \mathcal{E}^1_A : \sup \varphi=0\}$ and equip it with the subspace topology.

\subsubsection{The non-Archimedean Monge-Amp\`ere operator} The (non-Archimedean) Monge-Amp\`ere measure of a K\"ahler potential $\varphi_D$ is defined as $$\MA_A(\varphi_D):=V^{-1}\sum_i \left((A+D)^n\cdot (b_iE_i)\right)\delta_{v_{E_i}},$$ where $\mathcal{X}_0=\sum_i b_iE_i$ is the zero divisor on a SNC model $\mathcal{X}$ where $D$ is defined. This is easily seen to be a probability measure. There is also a natural extension of the non-Archimedean Monge-Amp\`ere operator to $\mathcal{E}^1_A(X^{\NA})$ (see Section \ref{sec:enpair}).  

\subsubsection{Finite energy measures}
The (dual) energy $\Ev_{A}(\mu)$ of a Radon probability measure is defined as $$\Ev_{A}(\mu):=\sup\left\{E_{A}(\varphi)-\int\varphi\, \mathrm d\mu: \varphi\in \mathcal{E}^1_{A}(X^{\NA})\right\},$$ those with $\Ev_{A}(\mu)<\infty$ giving us the space of finite energy measures $\mathcal{M}^1_{A}(X^{\NA})$. We endow $\mathcal{M}^1_{A}(X^{\NA})$ with the strong topology, defined as the coursest topology, finer than the weak topology of measures, that makes $\Ev_A$ continuous.

\subsubsection{A non-Archimedean version of Theorem \ref{thm:CYweak}}
Our second main result is a non-Archimedean Calabi--Yau Theorem, first proved in the algebraic case by Boucksom-Jonsson  \cite[Theorem 12.8]{BJ22trivval}. 

\begin{theorema} \label{thm:B}
The non-Archimedean Monge-Amp\`ere operator is a homeomorphism between $\mathcal{E}^1_{A,\sup}(X^{\NA})$ and $\mathcal{M}^1_{A}(X^{\NA})$.
\end{theorema}

We also provide an essential regularity result which says that if $\mu$ furthermore has support on the dual complex of some SNC model, then the solution $\varphi$ is continuous (see Theorem~\ref{thm:continuity}).

The proof of Theorem \ref{thm:B} follows the variational approach described in \cite{BFJ15solution, BJ22trivval}. To carry this through we need to establish two crucial properties of envelopes in the transcendental setting: the Continuity of Envelopes Property and the Orthogonality Property.

\addtocontents{toc}{\SkipTocEntry}
\subsection{Continuity of Envelopes} 

The $A$-psh envelope of a continuous function $f$ on $X^{\NA}$ is defined as $$P_A(f) :=\sup\{\varphi\in PSH_A(X^{\NA}): \varphi\leq f\}.$$

In Section \ref{sec:CoEPOP} we prove what is commonly known as the Continuity of Envelopes Property:
\begin{theorem} \label{thm:CoE}
For any continuous function $f$ on $X^{\NA}$ the $A$-psh envelope $P_A(f)$ is continuous and $A$-psh.
\end{theorem}

The proof relies on a correspondence between $A$-psh functions and subgeodesic rays (see Section \ref{sec:corr}), established in \cite{BJ21nakstabi} in the algebraic case and extended to the transcendental case in \cite{MP24transcendental}. 

\addtocontents{toc}{\SkipTocEntry}
\subsection{Orthogonality}

In Section \ref{sec:orthogonality} we prove what is commonly known as the Orthogonality Property:
\begin{theorem} \label{thm:Orth}
If $f$ is a continuous function on $X^{\NA}$ then
\begin{equation} \label{eq:OP}
\int(f-P_A(f))\MA_A(P_A(f))=0.
\end{equation}
\end{theorem}

The idea of our proof is to first consider the PL case $f=f_D$, furthermore assuming that $D\ge \cX_0=\sum_ib_iE_i$ (i.e. $f_D\ge 1$). It follows that $A+D$ is big, and key to our proof is the following explicit formula: 
\begin{equation} \label{eq:MAform}
\MA_A(P_A(f_D)) =V^{-1}\sum_ib_i\langle(A+D)^n\rangle_{\cX|E_i}\delta_{v_{E_i}},
\end{equation}
where $\langle(A+D)^n\rangle_{\cX|E_i}$ denotes the \emph{restricted volume} of $A+D$ along $E_i$ (see Section \ref{sec:restricted}). This formula was first proved in the algebraic case by Chi Li \cite{Li23fujita} and our proof is basically the same. In fact, using the definitions it is not difficult to see that $$\MA_A(P_A(f_D)) \ge V^{-1}\sum_ib_i\langle(A+D)^n\rangle_{\cX|E_i}\delta_{v_{E_i}},$$ and the equality then follows from the fact that
\begin{equation} \label{eq:voleq}
\sum_ib_i\langle(A+D)^n\rangle_{\cX|E_i}=V,
\end{equation}
meaning that both the LHS and the RHS are probability measures.  The equality (\ref{eq:voleq}) is a consequence of the differentiability of the volume function (in divisorial directions), proved in the algebraic case by Boucksom--Favre--Jonsson \cite{BFJ09differentiability} and Lazarsfeld--Mustaţă \cite{LM09convex} and in the transcendental case by the second named author \cite{WN24deformations} (see also \cite{Vu23derivative}). We then show that $f_D(v_{E_i})-P_A(f_D)(v_{E_i})\le b_i^{-1}\nu_{E_i}(A+D),$ where $\nu_{E_i}(A+D)$ denotes the Lelong number of $A+D$ along $E_i$ (see Section \ref{sec:lelong}). Since $\langle(A+D)^n\rangle_{\cX|E_i}=0$ whenever $\nu_{E_i}(A+D)>0$, this finally proves (\ref{eq:OP}) for $f=f_D$. The general case then follows by uniform approximation.

It is not very hard to see that one can find $D$ with prescribed restricted volumes $\langle(A+D)^n\rangle_{\cX|E_i}$ as long as (\ref{eq:voleq}) holds. Thus by (\ref{eq:MAform}) we can directly solve the Monge-Amp\`ere equation $\MA_A(\varphi)=\mu$ for divisorial measures $\mu\in \cM^{div}$, i.e. probability measures of the form $\sum_{i=1}^N a_i\delta_{v_{E_i}}$. This clearly shows the strong link that exists between properties of restricted volumes and the solvability of the non-Archimedean Monge-Amp\`ere equation.

\addtocontents{toc}{\SkipTocEntry}
\subsection{Related works} 
Many related works have already been mentioned in the introduction. Boucksom-Jonsson has recently announced a very general Yau--Tian--Donaldson correspondence between the existence of \emph{weighted} constant scalar curvature Kähler metrics and a \emph{weighted} version of K-stability for models, see \cite{BJ25YTD}. 
E.g., in the smooth unweighted algebraic case, they prove that the existence of a cscK metric implies K-stability for models.

Near the completion of this paper we were informed by Tamás Darvas and Kewei Zhang about their soon to appear work on a different kind of Yau--Tian--Donaldson correspondence, see \cite{DZ25YTD}. They define a notion of stability that they call \emph{$K^\beta$-stability}, which depends on a parameter $\beta\in \R$, and they prove that there exists a constant scalar curvature metric in $\alpha$ if and only if $(X,\alpha)$ is $K^\beta$-stable for some $\beta>0$. As their proof relies on transcendental methods, it applies in the general K\"ahler setting.

\addtocontents{toc}{\SkipTocEntry}
\subsection{Acknowledgements}

The authors would like to thank Robert Berman, Sébastien Boucksom, Tamás Darvas, Ruadhaí Dervan, Mattias Jonsson, Chung-Ming Pan, Rémi Reboulet, Julius Ross and Mingchen Xia for many insightful discussions about non-Archimedean pluripotential theory and transcendental methods.

In addition, the first named author is especially grateful to his PhD advisors Sébastien Boucksom and Tat Dat Tô for their helpful feedback.
The idea of considering the Legendre transform in connection to the valuative criterion came from a conversation with them, which we thank them for. 

The second named author has received funding from the Göran Gustafsson Foundation for Research in Natural Sciences and Medicine.
The first named author has received funding from the European Union’s Horizon 2020 research and innovation programme under the Marie Skłodowska-Curie grant agreement No 94532. \includegraphics*[scale = 0.029]{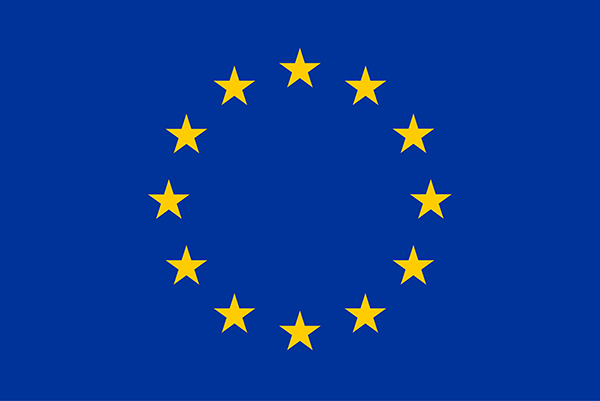}

\section{Preliminaries}

\addtocontents{toc}{\SkipTocEntry}
\subsection{Big cohomology classes}

In this section we briefly recall some pluripotential theory of big $(1,1)$-classes. As in the introduction we let $(X,\omega)$ be a compact K\"ahler manifold of dimension $n$ and $\alpha:=\{\omega\}$. We let $\beta\in H^{1,1}(X,\R)$ be a big class.

\subsubsection{$\theta$-psh functions and closed positive currents with analytic/minimal singularities} Let $\theta\in\beta$ be a smooth form. An upper semicontinuous (usc) function $\varphi: X\to[-\infty,\infty)$ is said to be $\theta$-psh if $\theta+dd^c\varphi\geq0$ as a current, and the set of $\theta$-psh functions is denoted by $PSH_{\theta}(X)$. We say that $\varphi$ has analytic singularities if locally we can write $\varphi=c\log(\sum_i |f_i|^2)+g$ where the $f_i$:s are holomorphic and $g$ is bounded. If $\varphi,\psi\in PSH_{\theta}(X)$ we say that $\varphi$ is less singular than $\psi$ if $\varphi\geq \psi-C$ for some constant $C$, and $\varphi$ is said to have minimal singularities if $\varphi$ is less singular than all $\theta$-psh functions. It is easy to see that such functions always exist. We also say that a closed positive current $T=\theta+dd^c\varphi$ has analytic/minimal singularities if $\varphi$ has analytic/minimal singularities.

\subsubsection{Lelong numbers} \label{sec:lelong} 
Let $u$ be a psh function in a neighbourhood of $0\in \mathbb{C}^n$. Then the Lelong number of $u$ at $0$ is defined as $$\nu_0(u):=\liminf_{z\to 0}\frac{u(z)}{\log|z|}.$$ If $u$ rather is psh in a neighbourhood of a point $x\in X$ where $X$ is a complex manifold, then $\nu_x(u):=\nu_0(u\circ g^{-1})$ where $f$ is a local holomorphic chart centered at $x$. If $T$ is a closed positive current which locally near $x$ can be written as $T=dd^c u$, then we let $\nu_x(T):=\nu_x(u)$, and if $Z\subseteq X$ is a subvariety we let $\nu_Z(T):=\inf_{x\in Z}\nu_x(T)$. 
We will later have use of the following well-known lemma:
\begin{lemma} \label{lem:suplelong}
Let $(u_i)_{i\in I}$ be a family of psh functions that are uniformly bounded from above on some fixed neighbourhood of $x\in X$, and let $u:=(\sup_i(u_i))^{\star}$, where $^{\star}$ denotes the upper semicontinuous (usc) regularization. Then $u$ is psh and $$\nu_x(u)=\inf_i(\nu_x(u_i)).$$
\end{lemma}
\begin{proof}
That $u$ is psh is a standard result in pluripotential theory, and the statement about Lelong numbers is a direct consequence of the elementary fact that if $v$ is subharmonic on the unit disc, $v\le 0$ and $\nu_0(v)\ge 1$, then $v\le \log|z|$.
\end{proof}

Given a big class $\beta$ we also let $$\nu_Z(\beta):=\inf\{\nu_Z(T): T\text{ is a closed positive current in }\beta\}.$$ If $T\in\beta$ has minimal singularities, then we have that $\nu_Z(\beta)=\nu_Z(T)$.
It is important to note that Lelong numbers depend continuously on the class (see e.g. \cite[Proposition 3.6]{Bou04divisorial}).

\subsubsection{$E_{nK}(\beta)$ and $E_{nn}(\beta)$}

A closed positive current $T=\theta+dd^c\varphi\in \beta$ is called a K\"ahler current if $T-\epsilon \eta\geq0$ for some $\epsilon>0$. We say that $x\in X$ lies in the K\"ahler locus of $\beta$ if there is a K\"ahler current $T\in \beta$ with analytic singularities which is smooth near $x$. The complement of the K\"ahler locus is called the non-K\"ahler locus of $\beta$ and is denoted by $E_{nK}(\beta)$.

The non-nef-locus is defined as $E_{nn}(\beta):=\{x\in X: \nu_x(\beta)>0\}$ and it is easy to see that $E_{nn}(\beta)\subseteq E_{nK}(\beta)$. It is also easy to show that if $x\in E_{nK}(\beta)\setminus E_{nn}(\beta)$ then $x$ lies in the K\"ahler locus of $\beta+\epsilon\alpha$ for any $\epsilon>0$.

\subsubsection{Positive products of currents} If $T_1,...,T_k$ are closed positive currents, following \cite{BEGZ10monge} one can form a closed positive $(k,k)$-current $\langle T_1 \wedge ... \wedge T_k\rangle$ known as the positive product. In the special case of $k=n$ and $T_i=T=\theta+dd^c\varphi$ for all $i$ we have that $\langle (\theta+dd^c\phi)^n\rangle$ is a positive measure which up to a constant is equal the non-pluripolar Monge-Amp\`ere measure $MA_{\theta}(\phi)$ (see \cite{BEGZ10monge}). 

\subsubsection{Positive intersections and volumes} \label{sec:posint} For $1\leq k\leq n$ the positive (or movable) intersection class $\langle \beta^k \rangle\in H^{k,k}(X,\R)$ is defined as $$\langle \beta^k \rangle:=\{\langle T^k\rangle\},$$ where $T$ is any closed positive current in $\beta$ with minimal singularities. An equivalent definition says that if $\gamma\in H^{n-k,n-k}(X,\R)$ is semipositive (i.e. contains a semipositive form) then $\langle \beta^k \rangle\cdot \gamma$ is the supremum of all numbers $(\beta')^k\cdot \mu^*\gamma$ where $\mu: X'\to X$ is a modification and $\beta'$ is a K\"ahler class on $X'$ such that $\beta'\leq\mu^*\beta$ (see \cite{Bou02cones}). In the special case of $k=n$ we get that $\langle\beta^n\rangle$ is a positive number, also known as the volume of $\beta$, written $\vol(\beta)$. If $\beta$ is K\"ahler then clearly $\langle\beta^k\rangle=\beta^k$.

\subsubsection{Restricted volumes} \label{sec:restricted}
Let $E$ be a prime divisor on $X$ which is not contained in $E_{nK}(\beta)$. The restricted volume of $\beta$ along $E$ is defined as $$\langle \beta^{n-1}\rangle_{X|E}:=\int_E\langle (T_{|E})^{n-1}\rangle,$$ where $T$ is any closed positive current in $\beta$ with minimal singularities. Equivalently it can be defined as the supremum of all numbers $(\beta')^{n-1}\cdot \tilde{E}$ where $\tilde{E}$ is the strict transform of $E$ under a modification $\mu: X'\to X$ and $\beta'$ is a K\"ahler class on $X'$ such that $\mu^*\beta-\beta$ is the class of an effective divisor $D$ whose support is contained in $\mu^{-1}(E_{nK}(\beta))$ (see \cite{ELMNP09restricted} and \cite[Theorem 5.3]{CT22restricted}). 
In the case when $E$ is contained in $E_{nK}(\beta)$ but not in $E_{nn}(\beta)$ we let $\langle \beta^{n-1}\rangle_{X|E}:=\lim_{\epsilon \to 0+}\langle (\beta+\epsilon\{\eta\})^{n-1}\rangle_{X|E}$, while if $E\subseteq E_{nn}(\beta)$ we let $\langle \beta^{n-1}\rangle_{X|E}:=0$. We say that $E$ is $\beta$-good if it either intersects the K\"ahler locus or lies in the non-nef locus of $\beta$. For a generic big class $\beta$, $E$ will be $\beta$-good.

It follows easily from the definitions that $$0\le\langle \beta^{n-1}\rangle_{X|E}\le \langle \beta^{n-1}\rangle \cdot E\le \langle (\beta_{|E})^{n-1}\rangle.$$

We also have the following key relationship between positive intersections and restricted volumes:

\begin{theorem} \label{thm:diffvol}
We have that $$\dtz\langle (\beta+tE)^n\rangle=n\langle \beta^{n-1}\rangle_{X|E}.$$
\end{theorem}

In the algebraic case this was proved by Boucksom--Favre--Jonsson \cite{BFJ09differentiability} and Lazars\-feld--Mustaţă \cite{LM09convex}, while in the transcendental setting this was proved by the second named author in \cite{WN24deformations}, assuming $E$ to be smooth and $\beta$-good. The general case was recently shown by Vu \cite{Vu23derivative}. 

An useful consequence of Theorem \ref{thm:diffvol} is the following (see \cite[Corollary A]{WN24deformations}):

\begin{corollary} \label{cor:WN}
If $E_1,...,E_m$ and $D_1,..., D_l$ are $\beta$-good prime divisors such that $$\sum_ia_i\{E_i\}=\sum_jb_j\{D_j\},$$ then we have that $$\sum_ia_i\langle \beta^{n-1}\rangle_{X|E_i}=\sum_jb_j\langle \beta^{n-1}\rangle_{X|D_j}.$$
\end{corollary}  

\addtocontents{toc}{\SkipTocEntry}
\subsection{A correspondence between subgeodesic rays and $A$-psh functions} \label{sec:corr}

An $S^1$-invariant $\pi_X^*\omega$-psh function $U$ on $X\times \mathds D$ which is bounded from above is called a \emph{subgeodesic ray}. Here the $S^1$-action is that coming from the standard action on the base. 

Note that if $(U_i)_{i\in I}$ is a family of subgeodesic rays that are uniformly bounded from above, then by standard pluripotential theory $(\sup_{i\in I}(U_i))^{\star}$ is also a subgeodesic ray.

Let $\cX\overset{\mu}{\to}X\times \Pro$ be a SNC model of $X$ with 
$\cX_0=\sum_i b_i E_i$. If $U$ is a subgeodesic ray we let 
$$U^{\NA}(v_{E_i}):= -b_i^{-1}\,\nu_{E_i}(U\circ \mu).$$

The next result was proved in the algebraic case in \cite[Section 6]{BBJ21YTD}, and the arguments were subsequently adapted to the transcendental setting in \cite[Theorem 5.1.4]{MP24transcendental}.

\begin{prop}\label{prop:extension}
The function $U^{\NA}$ on $\Xdiv$ defined above extends to an $A$-psh function on $X^{\NA}$.
\end{prop}

One can also go in the other direction, see \cite[Theorem 6.6]{BBJ21YTD} and \cite[Theorem 5.1.7]{MP24transcendental}.

\begin{prop}\label{prop:subgeodesic}
If $\varphi\in \cE_A^{1}(X^{\NA})$ with $\varphi\le 0$, then there is a subgeodesic ray $U_{\varphi}$ such that 
$U_{\varphi}^{\NA}=\varphi$.
\end{prop}

Given a function $f$ on $X^{\NA}$ we let $U_{\le f}$ be defined as the upper semicontinuous regularization of all subgeodesic rays $U\le 0$ such that $U^{\NA}\le f$. It follows from standard pluripotential theory results that $U_{\le f}$ itself is a subgeodesic ray.

\begin{prop} \label{prop:subgeodesicequal}
If $f\leq 0$ is usc, then $U_{\le \varphi}^{\NA}\le f$, and  for $\varphi\in \cE_A^{1}(X^{\NA})$ with $\varphi\le 0$ we get that $U_{\le \varphi}^{\NA}=\varphi$.
\end{prop}

Before proving this we recall the following result \cite[Theorem 4.3.7]{MP24transcendental}:  

\begin{prop}\label{prop:extineq}
If $\varphi\in PSH_A(X^{\NA})$, $f$ is usc and $\varphi\le f$ on $\Xdiv$, then $\varphi\le f$ on the whole of $X^{\NA}$.
\end{prop}

Now we can prove Proposition \ref{prop:subgeodesicequal}.

\begin{proof}
Let $v_{E_j}\in \Xdiv$ and let $(U_i)_{i\in I}$ be the family of all subgeodesic rays such that $U_i\le 0$ and $U_i^{\NA}\le f$. We then have that $$U_{\le f}^{\NA}(v_{E_j})=-b_j^{-1}\nu_{E_j}(U_{\le f})=-b_j^{-1}\inf_i\nu_{E_j}(U_i)=\sup_iU_i^{\NA}(v_{E_i})\le f(v_{E_i}),$$ where the second equality relied on Lemma \ref{lem:suplelong}. If $f=\varphi\in \cE^1_A(X^{\NA})$ then by Proposition \ref{prop:subgeodesic} there is a $U_i$ with $U_i^{\NA}=\varphi$, which shows that $U_{\le \varphi}=\varphi$.
\end{proof}

\addtocontents{toc}{\SkipTocEntry}
\subsection{Energy functionals on $\mathcal{E}^1_A(X^{\NA})$ and $\mathcal{M}^1_A(X^{\NA})$} 

To simplify notation we will henceforth instead of $PL(X^{\NA})$, $\mathcal{E}^1_A(X^{\NA})$ and $\mathcal{M}^1_A(X^{\NA})$ just write $PL$, $\mathcal{E}^1_A$ and $\mathcal{M}^1_A$. We also let $H^{1,1}(X^{\NA}):=\varinjlim_{\cX} H^{1,1}(\cX)$, so in particular $A:=\pi_X^*\alpha\in H^{1,1}(X^{\NA})$.

\subsubsection{The energy pairing} \label{sec:enpair}

Given $B_0, \dotsc, B_n\in H^{1,1}(X^{\NA})$ and $f_{D_0}, \dotsc, f_{D_n}\in \PL$ we define their energy pairing as:
\[(B_0, f_{D_0})\cdot (B_1, f_{D_1})\cdots (B_n, f_{D_n}):= (B_0 + D_0)\cdot (B_1 + D_1)\cdots (B_n + D_n) \in \R\]
where the intersection is done on a model $\cX$ on which all the $B_i$:s and $D_i$:s are determined.

We will later need the following estimate.

\begin{lemma} \label{lem:enpairest}
For given $B_l, \dotsc, B_n\in H^{1,1}(X^{\NA})$ there is a constant $C$ such that if $\varphi,\psi\in \cH_A$ with $|\varphi-\psi|\le \epsilon$, then $$|(A,\varphi)^l\cdot (B_l,0)\cdots (B_n,0)-(A,\psi)^l\cdot (B_l,0)\cdots (B_n,0)|\le C\epsilon.$$
\end{lemma}
\begin{proof}
By writing each $B_i$ as the difference of two semipositive classes we can reduce to the case to when all the $B_i$:s are semipositive. We write $\varphi=\varphi_D$ and $\psi=\varphi_{D'}$. Assume that $\varphi\ge \psi$. As this translates into $D-D'\ge 0$ we get that $$\left((A+D)^l-(A+D')^l\right)\cdot B_l \cdots B_n=(D-D')\cdot \left(\sum_{j=0}^l(A+D)^j\cdot (A+D')^{l-j}\right )\cdot B_l \cdots B_n\ge 0,$$ which shows that in this case, the energy pairing is monotone. It follows easily from the monotonicity that for general $\varphi,\psi\in \cH_A$ such that $|\varphi-\psi|\le \epsilon$ we get that
\begin{eqnarray*}
|(A,\varphi)^l\cdot (B_l,0)\cdots (B_n,0)-(A,\psi)^l\cdot (B_l,0)\cdots (B_n,0)|\le \\ \le(A+D+2\epsilon\cX_0)^l\cdot B_l\cdots B_n-(A+D)^l\cdot B_l\cdots B_n=\\=2\epsilon \cX_0\cdot \left(\sum_{j=0}^l(A+D+2\epsilon \cX_0)^j\cdot (A+D)^{l-j}\right )\cdot B_l \cdots B_n=2\epsilon(l+1) \cX_0 \cdot A^l\cdot B_l\cdots B_n,
\end{eqnarray*}
which establishes the claim.
\end{proof}

The Monge–Ampère energy $\E_A$ of a K\"ahler potential $\phi$ defined in the introduction thus corresponds to
\[\E_A(\varphi)= \frac{V^{-1}}{n+1}(A, \varphi) \cdots (A, \varphi).\] Given $\zeta \in H^{1,1}(X)$ we also define the \emph{twisted Monge–Ampère energy} as
\[\E^\zeta_A(\varphi):=V^{-1}( \pi_X^*\zeta,0)\cdot (A, \varphi)^n.\]

As shown in \cite{MP24transcendental}
the energy pairing extends to $(H^{1,1}(X^{\NA}))^{n+1}\times(\cE^1_A)^{n+1}$, and thus $\E_A$ and $\E^{\zeta}_A$ both extend to $\cE^1_A$.

If $\varphi \in \cH_A$, then the map
\[\PL\ni f \mapsto V^{-1}(0,f)\cdot (A, \varphi)^n\in \R\] extends by density to $C(X^{\NA})$ and thus defines a Radon measure on $X^{\NA}$. This measure is easily seen to coincide with the \emph{Monge–Ampère measure} $MA_A(\varphi)$ of $\varphi$ described in the introduction. It is verified in \cite{MP24transcendental} that the extension of the energy pairing in this way also allows one to extend the Monge-Amp\`ere operator to $\cE^1_A$.

As in the classical setting we have a Chern--Levine--Nirenberg inequality which says that for for any $\phi,\varphi, \psi\in \cE^1_A$:
\begin{equation}\label{eq:CLN}
    \left \lvert\int \phi \left(\MA_A(\varphi) - \MA_A(\psi)\right)\right\rvert \le n \sup \left\lvert \varphi - \psi\right\rvert.
\end{equation}
Indeed this follows exactly as in \cite[Equation 1.19]{BJ23synthetic}.

As a consequence we get the following continuity result for the Monge-Amp\`ere operator:
\begin{prop} \label{prop:MAunicont}
If $\varphi_n\in \cE^1_A$ converge uniformly to $\varphi\in \cE^1_A$, then $\MA_A(\varphi_n)$ converge weakly to $\MA_A(\varphi)$.
\end{prop}
\begin{proof}
Since PL functions are dense in $C(X^{\NA})$, and any PL function can be written as the difference of two K\"ahler potentials, it is enough to verify that for every $\phi\in \cH_A$ we have that
    \[\left\lvert \int \phi \left(\MA_A(\varphi_n) - \MA_A(\varphi)\right)\right\rvert \to 0.\]
This evidently follows from (\ref{eq:CLN}).
\end{proof}

\subsubsection{Finite energy measures}
Recall that a Radon probability measure $\mu$ on $X^{\NA}$ has finite energy if 
\[\Ev_A(\mu) = \sup_{\varphi \in \cE^1_A} \left\{\E_A(\varphi) - \int\varphi \, \mathrm d\mu\right\}<+\infty.\]
In particular, if $\mu$ is of finite energy then any finite energy potential $\varphi \in \cE^1_A$ is integrable with respect to $\mu.$

As shown in \cite{BJ23synthetic} it follows from the Orthogonality Property that we prove in Section \ref{sec:CoEPOP} that as a topological space $\cM^1_A$ does not depend on the choice of class $A$. Therefore, for the rest of the paper we will just denote the set of finite energy measures by $\cM^1$.

\begin{example}\label{ex:delta}
    If $v\in \Xdiv$, then the Dirac mass $\delta_v$ is of finite energy.

    Indeed, the energy of $\delta_v$ is given by 
    \begin{align*}
        \Ev_A(\delta_v) =\sup_{\varphi\in \cE^1_A}\left\{\E_A(\varphi) - \int \varphi \, \delta_v\right\} &=  \sup_{\varphi\in \cE^1_A}\{\E_A(\varphi) - \varphi(v)\}\\
        &\le \sup_{\varphi\in \cE^1_A}\{\sup \varphi - \varphi(v)\},
    \end{align*}
    which is finite by Lemma~\ref{lem:pshunifbounded} below.
\end{example}

\begin{lemma}\label{lem:pshunifbounded}
For any two divisorial points $v,w\in \Xdiv$ there is a constant $C$ such that $\lvert \varphi(v) - \varphi(w)\rvert \le C$ for every $\varphi\in \PSH_A$. Since $\sup\varphi=\varphi(v_{\trivial})$ we in particular have that $\sup \varphi-\varphi(v)\leq C$.
\end{lemma}
\begin{proof}
This follows from the proof of \cite[Theorem 4.1.5]{MP24transcendental}.
\end{proof}

\begin{prop}
For every $\varphi\in \cE^1_A$ we have that $\MA_A(\varphi)\in \cM^1.$
\end{prop}
\begin{proof} As observed in \cite[Remark 4.4.5]{MP24transcendental}, based on the concavity of the Monge--Ampère energy of \cite[Lemma 1.19]{BJ23synthetic}, we have that for every $\psi\in \cE^1_A$:
    \[\E_A(\psi)\le \E_A(\varphi) + \int(\psi-\varphi)\MA_A(\varphi).\]

It immediately follows that $\Ev(\MA(\varphi)) = \E(\varphi) - \int\varphi\, \MA(\varphi),$ which in particular shows that $\Ev(\MA(\varphi))< +\infty.$
\end{proof}

Given $\varphi\in \cE^1_A$, following \cite{BJ23synthetic} we define the functional $\Jd_A(\cdot, \varphi)\colon \cM^1\to \left[0, +\infty \right[$ as 
\[\Jd_A(\mu, \varphi):= \Ev_A(\mu) - \E_A(\varphi) + \int \varphi\,\mathrm d\mu.\]

As a consequence of \cite[Proposition 2.2]{BJ23synthetic} we have the following basic properties for $\Jd_A(\cdot, \varphi).$
\begin{lemma}\label{lem:Jmu}
For any $\varphi, \psi\in \cE_A^1$ and $\mu\in \cM^1$ we have that
    \begin{enumerate}
        \item $\Jd_A(\MA_A(\psi), \varphi) = \Jd_A(\psi, \varphi),$
        \item $\Jd_A(\varphi, \psi)\lesssim \Jd_A(\mu, \varphi) + \Jd_A(\mu, \psi),$
        \item $\Jd_A(\mu, \cdot)$ is continuous under decreasing limits.
    \end{enumerate}
Here $x\lesssim y$ that there is a dimensional constant $C_n>0$ such that $x\le C_n y$ ($x\approx y$ means that $x\lesssim y$ and $y\lesssim x$).
\end{lemma}
\begin{proof}
The first point follows from the observation that $\Ev_A(\MA_A(\varphi)) = \E_A(\varphi) - \int \varphi \MA_A(\varphi)$.
For the second point we consider decreasing sequences $\varphi_j, \psi_j\in \cH_A$ converging to $\varphi$ and $\psi$ respectively.
From \cite[Proposition 2.2]{BJ23synthetic} we get that
\[\Jd_A(\varphi_j, \psi_j)\lesssim 2 \Ev_A(\mu) - \E_A(\varphi_j) - \E_A(\psi_j) + \int (\psi_j + \varphi_j)\, \mathrm d\mu,\]
and using the continuity of $\E_A$ along decreasing sequences and the monotone convergence theorem we then get
 \[\Jd_A(\varphi, \psi)\lesssim 2\Ev_A(\mu) - \E_A(\varphi) - \E_A(\psi) + \int(\varphi + \psi )\, \mathrm d \mu,\]
as desired.

The third point also follows from the continuity of $E_A$ along decreasing sequences and the monotone convergence theorem.
\end{proof}

\addtocontents{toc}{\SkipTocEntry}
\subsection{Entropy functionals and $\widehat{K}$-stability} 

\subsubsection{Log discrepancy and entropy}
For any $v_{E_i}\in \Xdiv$, where the prime divisor $E_i$ lives on a model $\cX\overset{\mu}{\longrightarrow} X\times \Pro$, we define the \emph{log discrepancy of $v_{E_i}$} as
\[A_{X\times \Pro}(v_{E_i}) := b_i^{-1} A_{X\times \Pro}(E_i)=b_i^{-1}\left(1 + \ord_{E_i}(K_{\cX/X\times \Pro})\right).\]
By an approximation procedure described in \cite{MP24transcendental}, itself based on \cite{JM12val}, we can extend the log discrepancy function $A_{X\times\Pro}$ from $\Xdiv$ to $X^{\NA}$.

\begin{defi}
We define the entropy of a Radon measure $\mu$ on $X^{\NA}$ as
\[\Ent(\mu):= \int (A_{X\times\Pro} -1) \, \mathrm d\mu,\] and the entropy of a finite energy potential $\varphi\in \cE^1_A$ as $\He_A(\varphi):=\Ent(\MA_A (\varphi))$. 
\end{defi}

\subsubsection{The non-Archimedean Mabuchi functional and $\widehat{K}$-stability} \label{sec:M_A}
Using the entropy functional we can now define the \emph{non-Archimedean Mabuchi functional} on $\mathcal{E}^1_A$  as
\[\M_A:= \underline{s}\E_A + \E_A^{K_X} + \He_A,\] where $\underline{s}:=-V^{-1}\alpha^{n-1}\cdot K_X$.

We also let \[\Jd_A(\varphi, \psi):= \E_A(\varphi) - \E_A(\psi) + \int(\psi- \varphi) \,\MA_A(\varphi),\] and $\Jd_A(\varphi):=\Jd_A(0,\varphi)$. 

\begin{defi}
    We say that $(X, \alpha)$ is \emph{uniformly  $\widehat K$-stable} if there exists a $\delta>0$ such that for every $\varphi \in \cE^1_A$ we have that
    \[\M_A(\varphi)\ge \delta \Jd_A(\varphi).\]
\end{defi}

The next result was first proved in the algebraic case by Chi Li \cite{Li22geodesic} and then extended to the general case by the first named author in \cite{MP24transcendental}.
\begin{theorem}
    If $(X,\alpha)$ is uniformly $\widehat K$-stable, then there is a unique cscK metric in $\alpha$.
\end{theorem}

\section{Weak topology on $\PSH_A$ and non-pluripolar points}

\addtocontents{toc}{\SkipTocEntry}
\subsection{Compactness of $\PSH_{A,\sup}$} We recall that $\PSH_A$ is endowed with the \emph{weak topology} of pointwise convergence on $\Xdiv$. We let $\PSH_{A, \sup}:=\left\{\varphi\in \PSH_A\mid \sup\varphi =0\right\}$. The goal of the first part of this section is to prove the following result.
\begin{theorem}\label{thm:compactness}
$\PSH_{A, \sup}$ is compact in the weak topology.
\end{theorem}

To accomplish this we will need the following proposition:

\begin{prop} \label{prop:regsupApsh}
For every family $(\psi_i)_{i\in I}$ of $A$-psh functions that is uniformly bounded from above, we have that $(\sup_i\psi_i)^{\star}$ is $A$-psh, and furthermore that $\sup_i\psi_i=(\sup_i\psi_i)^{\star}$ on $\Xdiv$.
\end{prop}

\begin{proof}
Without loss of generality we can assume that each $\psi_i\le 0$. We also start by assuming that each $\psi_i\in \cE^1_A$.

Let $U_i:=U_{\le \psi_i}$, where is defined as in Section \ref{sec:corr}. Recall that by Proposition \ref{prop:subgeodesicequal} $U_i^{\NA}=\psi_i$. Also let $V:=(\sup_i U_i)^{\star}$ and $\psi:=V^{\NA}\in PSH_A.$ By monotonicity of Lelong numbers we get that for all $i$: $\psi\ge \psi_i$ on $\Xdiv$, and by Proposition \ref{prop:extineq} we get that $\psi\ge \psi_i$ on the whole of $X^{\NA}$. Thus $\psi\ge \sup_i\psi_i$, and since $\psi$ is usc we get that $\psi\ge (\sup_i\psi_i)^{\star}$. If $v_{E_j}\in \Xdiv$ we now have that $$\psi(v_{E_j})=-b_j^{-1}\nu_{E_j}(V)=-b_j^{-1}\sup_i\nu_{E_j}(U_i)=\sup_i U_i^{\NA}(v_{E_i})=\sup_i \psi_i(v_{E_i}),$$ which clearly implies that $$(\sup_i\psi_i)^{\star}(v_{E_i})=\sup_i \psi_i(v_{E_i})\le (\sup_i \psi_i)^{\star}(v_{E_i}).$$ This implies that $\psi=(\sup_i\psi_i)^{\star}$ on $\Xdiv$. It now follows from Proposition \ref{prop:extineq} that $\psi\le (\sup_i\psi_i)^{\star}$ on the whole of $X^{\NA}$, and hence $\psi=(\sup_i\psi_i)^{\star}$. This concludes the proof in the case where each $\psi_i\in \cE^1_A$. 

In the general case, we observe that for any constant $C$: $$\max((\sup_i \psi_i)^\star,C)=(\sup_i \max(\psi_i,C))^\star,$$ and $\max(\psi_i,C)\in \cE^1_A$. Thus $(\sup_i\psi_i)^{\star}$ is the decreasing limit of $A$-psh functions and hence $A$-psh. We similarly have that $\max((\sup_i \psi_i),C)=\sup_i (\max(\psi_i,C)),$ which lets us conclude that $\sup_i\psi_i=(\sup_i\psi_i)^{\star}$ on $\Xdiv$ also in the general case.

\end{proof}

We are now ready to prove Theorem \ref{thm:compactness}, following the argument in \cite[Theorem 5.11]{BJ22trivval}.

\begin{proof}[Proof of Theorem~\ref{thm:compactness}] 
    
Let $\varphi_i\in \PSH_{A,\sup}$ be a family of sup normalized $A$-psh functions.

If $x\in \Xdiv$ we get by Lemma~\ref{lem:pshunifbounded} that the sequence of real numbers $ \varphi_i(x)$ is bounded. Thus by Tychonoff's theorem there exists a  subsequence that converges to a function $\varphi\colon \Xdiv\to \R$, and we will let $\varphi_i$ denote the subsequence.
    
Let $\psi_i:=(\sup_{j\ge i} \varphi_j)^\star$. By Proposition~\ref{prop:regsupApsh}  $\psi_i$ is $A$-psh, and by definition $(\psi_i)_i$ is decreasing. Also by Proposition~\ref{prop:regsupApsh} $\psi_i=\sup_{j\ge i}\varphi_j$ on $\Xdiv$. In particular $\sup_{X^{\NA}} \psi_i=\psi_i(v_{triv})=\sup_{j\ge j}\varphi_j(v_{triv})=0$, and hence $\psi:=\lim_i\psi_i\in PSH_{A,\sup}$.
Since $\varphi_j(x)\to \varphi(x)$ for $x\in \Xdiv$ we clearly also get that $\psi(x)=\varphi(x)=\lim_j \varphi(x)$, i.e. that $\varphi_j$ converges weakly to $\psi$.
\end{proof}

\addtocontents{toc}{\SkipTocEntry}
\subsection{Non-pluripolar points} \label{sec:np}

In this Section we will follow \cite[Sections 4.5, 4.6 and 11]{BJ22trivval} and consider the set of \emph{non-pluripolar} points of $X^{\NA}$, .

\begin{defi}
    We say that $v\in X^{\NA}$ is \emph{non-pluripolar} if for every $\varphi\in \PSH$
    \[\varphi(v)>-\infty.\]
    We denote the set of non-pluripolar points by $\Xnp.$
\end{defi}

By Theorem 4.3.3 of \cite{MP24transcendental} we have that $\Xdiv\subseteq \Xnp$.

For $v,w\in \Xnp$ we let $$d_{\infty}(v,w):=\sup_{\varphi\in \PSH} \lvert \varphi(v) - \varphi(w)\rvert.$$
We also let $\mathrm T(v):= d_\infty(v, v_{\trivial}).$

\begin{prop}
    $d_{\infty}$ is a metric on $\Xnp$.
\end{prop}
\begin{proof}
It is clear that for all $v,u,w\in \Xnp$: $d_\infty(v,w) = d_\infty(w,v)$ and $d_\infty(v,w)\le d_\infty(v,u) + d_\infty(u,w)$.
To show that $d_\infty$ is finite valued it is then enough to check that $d_\infty(v, v_{\trivial})<\infty$ for $v\in \Xnp$. For that, suppose that by contradiction that for some $v$ there exists a sequence $\varphi_m\in \PSH$ such that
    \[\lvert \varphi_m(v_{\trivial}) - \varphi (v)\rvert >2^m.\]
    By adding a constant we may suppose that $\varphi_m(v_{\trivial}) = 0$, and therefore the convex combination
    $\psi_m:= 2^m \cdot 0 +\sum_{k=1}^m 2^{-k}\varphi_k\in \PSH$ is psh, decreasing with $m$, and satisfies:
    \[\psi_m(v_{\trivial}) = 0, \quad \text{and } \psi_m(v) \le \sum_{k=1}^m 2^{-k}\cdot (-2^k) =-m.\]
    Therefore $\psi:=\lim \psi_m\in PSH_A$ and $\psi(v) = -\infty$.

    Lastly, if $d_\infty(v,w) = 0$, then for every $\varphi\in\cH$:
    \[\varphi(v) = \varphi(w).\]
    Since the linear span of $\cH$ is $\PL$, which is dense in $C(X^{\NA})$, this implies that $v=w$.
\end{proof}

Following the same strategy as in the previous proposition we obtain the next result. For more details see \cite[Proposition 11.1]{BJ22trivval}.
\begin{prop}\label{prop:Tinequality}
    For every $v\in \Xnp$ we have:
    \[\frac{1}{n+1}\T(v)\le \Ev(\delta_v)\le \T(v).\]
\end{prop}

\begin{example}\label{ex:bounded}
    Let $\Sigma:=\{v_1, \dots, v_k\}$ be a finite set of divisorial valuations, then $\Sigma$ is bounded for $d_\infty$.
    Indeed, this follows directly from Lemma~\ref{lem:pshunifbounded}. 
\end{example}

\section{Continuity of envelopes and orthogonality} \label{sec:CoEPOP}
\addtocontents{toc}{\SkipTocEntry}
\subsection{Continuity of envelopes}

The goal of this section is to prove Theorem \ref{thm:CoE}, commonly known as the Continuity of Envelopes Property, which says that if $f$ is a continuous function on $X^{\NA}$, then the $A$-psh envelope
\[P_A(f):= \sup\left\{\varphi\in \PSH_A \mid \varphi\le f\right\}\] is continuous and $A$-psh.

\begin{proof}[Proof of Theorem \ref{thm:CoE}]
We start by showing that $P_A(f)$ is lower semicontinuous. Here we follow the argument in \cite[Lemma 5.17]{BJ22trivval}. 

Let $\psi$ be $A$-psh such that $\psi\le f$, and let $\varphi_i$ be a net in $\cH_A$ which decreases to $\psi$. Then since $X^{\NA}$ is compact and $f$ is continuous, for every $\epsilon>0$ we can find a $\varphi_i$ such that $\varphi_i-\epsilon\le f$. This shows that $$P_A(f)=\sup\{\varphi\in \cH_A \mid \varphi\le f\},$$ and hence $P_A(f)$, being the supremum of continuous functions, is lower semicontinuous.

Thus we are left to prove that $P_A(f)$ is $A$-psh, since it is then also upper semicontinuous.

Let $U_{\le f}$ be the subgeodesic ray defined as in Section \ref{sec:corr}. By Proposition \ref{prop:extension} we have that $U_{\le f}^{\NA}\in PSH_A$ and $U_{\le f}^{\NA}\le f$. This means that $U_{\le f}^{\NA}$ is a candidate for the envelope $P_A(f)$ and thus $U_{\le f}^{\NA}\le P_A(f)$. 

We now want to prove the reverse inequality, as that would show that $P_A(f)=U_{\le f}^{\NA}\in PSH_A$.

Let $\varphi\in \cH_A$ be such that $\varphi\le f\le 0$. By Proposition \ref{prop:subgeodesicequal}  $U_{\le\varphi}^{\NA}=\varphi\le f$, and thus $U_{\le\varphi}$ is a candidate for the envelope $U_{\le f}$. Hence $U_{\le\varphi}\le U_{\le f}$, and by the monotonicity of Lelong numbers we get that $\varphi=U_{\le\varphi}^{\NA}\le U_{\le f}^{\NA}$ on $\Xdiv$. Since $U_{\le f}^{\NA}-\varphi$ is usc the inequality extends to $X^{\NA}$. This finally proves that $P_A(f)\le U_{\le f}^{\NA}$, and hence that $P_A(f)=U_{\le f}^{\NA}\in PSH_A$.
\end{proof}

\addtocontents{toc}{\SkipTocEntry}
\subsection{Orthogonality} \label{sec:orthogonality}

The goal of this section is to prove Theorem \ref{thm:Orth}, commonly known as the Orthogonality Property, which says that if $f$ is a continuous function on $X^{\NA}$, then 
\begin{equation} \label{eqn:orth}
\int (f-P_A(f))\MA_A(P_A(f))=0.
\end{equation}

The idea of the proof is to first show (\ref{eqn:orth}) in the special case of $f=f_D\in PL$, and then use that any $f\in C(X^{\NA})$ can be uniformly approximated by PL functions.

Let $D$ be a vertical divisor on a SNC model $\cX$ such that $D\ge \cX_0$. We write $\cX_0=\sum_ib_iE_i$ and as usual let $f_D$ denote the PL function associated to $D$. Note that $D\ge \cX_0$ is the same as saying that $f_D\ge 1$, and we denote the set of such PL functions by $PL_{\ge 1}$. 

We start with the following fact, which will turn out to be crucial:

\begin{prop} \label{prop:sumvol}
We have that $$\sum_i b_i\langle (A+D)^n\rangle_{\cX|E_i}= V.$$
\end{prop}

\begin{proof}
First we note that $D\ge \cX_0$ implies that $A+D$ is big. Let $D'$ be an effective vertical divisor such that $A+D'$ is K\"ahler and let $B_{\epsilon}:=(1+\epsilon)A+D+\epsilon D'$. Then for every small enough $\epsilon>0$ we have that each of the prime divisors $E_i$ and $\cX_1$ are $B_{\epsilon}$-good, and we also have that $$\sum_ib_i\{E_i\}=\{\cX_0\}=\{\cX_1\}.$$ Thus by Corollary \ref{cor:WN} we get that  
\begin{eqnarray*}\sum_i b_i\langle (A+D)^n\rangle_{\cX|E_i}=\lim_{\epsilon\to 0+}\sum_ib_i\langle B_{\epsilon}^n\rangle_{\cX|E_i}=\lim_{\epsilon\to 0+}\langle B_{\epsilon}^n\rangle_{\cX|\cX_1}=\langle(A+D)^n\rangle_{\cX|\cX_1}.
\end{eqnarray*}
We then observe that $$V=A^n\cdot \cX_1\le \langle(A+D)^n\rangle_{\cX|\cX_1}\le \langle (A+D)_{|\cX_1}^n\rangle=A^n\cdot \cX_1=V,$$ giving us the result.
\end{proof}

The next step is to show the following transcendental version of \cite[Theorem 1.1 (i)]{Li23fujita}.
\begin{prop} \label{prop:MAform}
For $f_D\in PL_{\ge 1}$ we have that
$$\MA_A(P_A(f_D)) =V^{-1}\sum_ib_i\langle(A+D)^n\rangle_{\cX|E_i}\delta_{v_{E_i}}.$$
\end{prop}

\begin{proof}
The proof is similar to that in \cite{Li23fujita}.
Since we can write $A+D$ as the sum of a K\"ahler class and (the class of) an effective vertical divisor we clearly get that $E_{nK}(A+D)\subseteq \cX_0$.

Assume that $E_i$ is not contained in $E_{nK}(A+D)$. Then for any $k\in \mathbb{N}$ there is a dominating model $\mu_k:\cX^k\to \cX$ whose center does not contain $E_i$, a K\"ahler class $A_k$ and an effective divisor $D_k$ supported on $E_{nK}(A+D)\subseteq \cX^k_0$ (we here identify $A+D$ with its pullback to $\cX^k$) such that $A+D=A_k+D_k$ and $$(A_k)^n\cdot \tilde{E}_i\ge \langle (A+D)^n\rangle_{\cX|E_i}-1/k.$$ Here $\tilde{E}_i$ denotes the strict transform of $E_i$. We also note that $D_k$ is vertical and hence $\varphi_k:=\varphi_{D-D_k}\in \cH_A$, and since $D_k\ge 0$ we have that $\varphi_k\le f_D$.

By the Continuity of Envelopes Property established in the previous section we have that $P_A(f_D)$ is continuous and $A$-psh and we can thus find an increasing sequence of K\"ahler potentials $\psi_k$ that converge uniformly to $P_A(f_D)$.

Now we let $$\varphi_k'=\varphi_{D-D_k'}:=\max(\varphi_1,...,\varphi_k,\psi_k)\in PL\cap PSH_A$$ and $A'_k:=A+D-D'_k$. By \cite[Theorem 4.1.5]{MP24transcendental} $A'_k$ is relatively nef, and by approximation we can without loss of generality assume that $A'_k$ is relatively K\"ahler. Since $A'_k-A_k=D_k-D'_k$ is effective and $\tilde{E}_i$ is not contained in its support, we have that $$MA_A(\varphi_k)(v_{E_i})=V^{-1}(A_k')^n\cdot \tilde{E}_i\ge V^{-1}(A_k)^n\cdot \tilde{E}_i\ge V^{-1}\langle (A+D)^n\rangle_{\cX|E_i}-V^{-1}/k.$$ Since $\varphi_k'$ also converges uniformly to $P_A(f_D)$ we get that $MA_A(\varphi_k)$ converges weakly to $MA_A(P_A(f_D))$ and hence $$MA_A(P_A(f_D))(\{x_i\})\ge V^{-1}\langle (A+D)^n\rangle_{\cX|E_i}.$$

If instead $\nu_{E_i}(A+D)>0$ we trivially have that $$MA_A(P_A(f_D))(\{x_i\})\ge 0= V^{-1}\langle (A+D)^n\rangle_{\cX|E_i}.$$ 

Thus assuming each $E_i$ is $(A+D)$-good we get that $$MA_A(P_A(f_D)) \geq V^{-1}\sum_ib_i\langle(A+D)^n\rangle_{\cX|E_i}\delta_{x_i}.$$ On the other hand we get from Proposition \ref{prop:sumvol} that the RHS is a probability measure just as the LHS, which implies that they are equal.

If some $E_i$ fail to be $\beta$-good we let $D'\ge D$ be such that $A+D'$ is K\"ahler and let $D_{\epsilon}:=(D+\epsilon D')/(1+\epsilon)$. We also let $\psi_{\epsilon}:=P_A(f_{D_{\epsilon}})$, and note that $\psi_{\epsilon}$ converge uniformly to $P_A(f_D)$ as $\epsilon \to 0+$. For small enough $\epsilon$ each $E_i$ is $(A+D_{\epsilon})$-good, and thus using Proposition \ref{prop:MAunicont} we get that
\begin{eqnarray*}
MA_A(P_A(f_D))=\lim_{\epsilon \to 0+}MA_A(\psi_{\epsilon})= \lim_{\epsilon \to 0+}V^{-1}\sum_ib_i\langle(A+D_{\epsilon})^n\rangle_{\cX|E_i}\delta_{x_i}=\\ V^{-1}\sum_ib_i\langle(A+D)^n\rangle_{\cX|E_i}\delta_{x_i}.
\end{eqnarray*}
\end{proof}

\begin{prop} \label{prop:lelongineq}
We have that $$f_D(v_{E_i})-P_A(f_D)(v_{E_i})\le b_i^{-1}\nu_{E_i}(A+D).$$
\end{prop}
\begin{proof}
By Demailly regularization, for every $\epsilon>0$ we can find a K\"ahler current $T\in A+D$ with analytic singularities such that $\nu_{E_i}(T)\le \nu_{E_i}(A+D)-\epsilon$. Thus on some model $\cX$, $T$ will have divisorial singularities, and since $E_{nK}(A+D)\subseteq \cX_0$ we can assume that the singularities are supported on $\cX_0$. Hence we can write $T=\Omega+[D']$ where $\Omega$ is K\"ahler and $D'$ is an effective vertical divisor. It follows that $\varphi:=\varphi_{D-D'}\in \cH_A$ and $\varphi\le f_D$, which implies that $\varphi\le P_A(f_D)$. We finally get that $$f_D(v_{E_i})-P_A(f_D)(v_{E_i})\le f_D(v_{E_i})-\varphi(v_{E_i})= b_i^{-1}\nu_{E_i}([D'])=b_i^{-1}\nu_{E_i}(T)\le \nu_{E_i}(A+D)-\epsilon.$$ Since $\epsilon>0$ was arbitrary this proves the claim.
\end{proof}

We now prove Theorem~\ref{thm:Orth}.

\begin{proof}[Proof of Theorem~\ref{thm:Orth}]
We first assume that $f=f_D$ is PL. Since the integral $$\int (f_D-P_A(f_D))\MA_A(P_A(f_D))$$ is unaffected by adding a constant to $f_D$, without loss of generality we can assume that $f_D\in PL_{\ge 1}$. By Proposition \ref{prop:MAform} $\MA_A(P_A(f_D))$ is supported on the points $v_{E_i}$ corresponding to the irreducible components $E_i$ of the zero fiber $\cX_0$ on some fixed model $\cX$. Note that by definition $\langle(A+D_{\epsilon})^n\rangle_{\cX|E_i}=0$ if $\nu_{E_i}(A+D)>0$, so using Proposition \ref{prop:lelongineq} we see that $f_D-P_A(f_D)=0$ on the support of $\MA_A(P_A(f_D))$, proving the claim in the PL case.

For a general $f\in C(X^{\NA})$ we choose a sequence of PL functions $f_i$ converging uniformly to $f$. It is easy to see that $P_A(f_i)$ also will converge uniformly to $P_A(f)$, and thus by Proposition \ref{prop:MAunicont} $\MA_A(P_A(f_i))$ converges weakly to $\MA_A(P_A(f))$. This then shows that $$\int (f-P_A(f))\MA_A(P_A(f))=\lim_i\int (f_i-P_A(f_i))\MA_A(P_A(f_i))=0.$$
\end{proof}

\section{Strong topologies}
Having the Continuity of Envelopes and Orthogonality Propoerties we can develop more of the dual pluripotential theory of $\cE^1_A$ and of $\cM^1$. In this we will rely on various estimates from \cite{BJ23synthetic}.

\begin{prop} \label{prop:bjest1}
    Let $\gamma_0, \dotsc, \gamma_n\in H^{1,1}(X)$, and $\varphi_0, \dotsc \varphi_n, \psi_0, \dotsc, \psi_n \in \cE^1_A$, such that $\sup \varphi_i = 0= \sup \psi_i$ for every $i\in \{1, \dotsc, n\}$, then we have
    \[\left\lvert(\gamma_0, \varphi_0)\cdots (\gamma_n, \varphi_n) - (\gamma_0, \psi_0)\cdots (\gamma_n, \psi_n)\right\rvert\lesssim V(1+\lambda)^{n+1}\max_{i}\Jd_A(\varphi_i, \psi_i)^q\cdot \max_{i}\Jd_A(\varphi_i)^{1-q}\]
    for $q = 2^{-n}$ and $\lambda>0$ such that $-\lambda\alpha\le\gamma_i\le \lambda \alpha$ for every $i\in \{1, \dotsc, n\}$.
\end{prop}
\begin{proof}
    If $\varphi_i, \psi_i\in \cH_A$ this is an immediate consequence of \cite[Theorem 3.6]{BJ23synthetic}. For general $\varphi_i, \psi_i\in \cE^1_A$ we consider decreasing sequences $\varphi_{i,k}, \psi_{i,k}\in \cH_A$ converging to $\varphi_i$ and $\psi_i$ respectively and use that both sides of the inequality are continuous under decreasing sequences.
\end{proof}

\begin{corollary}\label{cor:MAestimate}
    Let $\varphi, \varphi^\prime, \psi, \psi^\prime \in \cE^1_A$ sup-normalized potentials of finite energy and let $\mu:=\MA_A(\psi)$ and $\nu:= \MA_A(\psi^\prime)$. We then have that:
    \[\left \lvert \int \varphi\, \mathrm d\mu - \int \varphi^\prime\, \mathrm d\nu\right\rvert \lesssim \max\{\Jd_A(\varphi, \varphi^\prime), \Jd_A(\psi, \psi^\prime)\}^q\cdot J^{1-q}, \]
    for $q:=2^{-n}$ and $J:=\max\{\Jd_A(\psi), \Jd_A(\psi^\prime), \Jd_A(\varphi), \Jd_A(\varphi^\prime)\}.$
\end{corollary}
\begin{proof}
    This follows directly from Proposition \ref{prop:bjest1}.
\end{proof}

This estimate can be used to prove the following continuity result for the Monge–Ampère operator:
\begin{lemma}\label{lem:MAcontdecreasing}
If $\varphi_i\in \cE^1_A$ is a decreasing sequence converging pointwise to $\varphi\in \cE^1_A$, then $MA_A(\varphi_i)$ converges strongly to $MA_A(\varphi)$ in $\cM^1$.
\end{lemma}
\begin{proof}
Let $\mu_i:=\MA_A(\varphi_i)$ and $\mu:=\MA_A(\varphi).$
Since the energy pairing is continuous along decreasing sequences we have that:
\begin{enumerate}
    \item $\lim_{i\to\infty}\E_A(\varphi_i)=\E_A(\varphi)$, $\lim_{i\to\infty}\Jd_A(\varphi_i)=\Jd_A(\varphi)$ and $\lim_{i\to\infty}\Jd_A(\varphi_i, \varphi)=0$.
    \item For every $f\in \PL$:
        \[\int f\,\mathrm d\mu_i \to \int f\,\mathrm d\mu.\]
\end{enumerate}
Since PL is dense in $C(X^{\NA})$, (2) implies that $\mu_i$ converges weakly to $\mu$.

To get the strong convergence it is enough to observe that 
\[\Ev_A(\mu_i) = \E_A(\varphi_i) - \int \varphi_i \,\mathrm d\mu_i \quad \text{and} \quad \Ev_A(\mu) = \E_A(\varphi) - \int \varphi \,\mathrm d\mu.\]
Thus by Corollary~\ref{cor:MAestimate}
\begin{eqnarray*}
    \left\lvert \Ev_A(\mu_i) - \Ev_A(\mu)\right\rvert \le \\ \le \left\lvert \E_A(\varphi_i) - \E_A(\varphi)\right\rvert + \left\lvert \sup \varphi_i - \sup\varphi\right\rvert
     +\left\lvert \int (\varphi_i-\sup \varphi_i)\,\mathrm d\mu_i - \int (\varphi-\sup \varphi)\,\mathrm d\mu\right\rvert
    \lesssim \\ \lesssim \left\lvert \E_A(\varphi_i) - \E_A(\varphi)\right\rvert + \left\lvert \sup \varphi_i - \sup\varphi\right\rvert
    + C\Jd_A(\varphi_i, \varphi)^q,
    \end{eqnarray*}
    for $q:=2^{-n},$ and some constant $C>0$ independent of $i$. Together with (1) and (2) this gives the convergence needed.
\end{proof}

The next result is a direct consequence of \cite[Theorem 2.23 (iv)]{BJ23synthetic}.
\begin{prop}\label{prop:2.23}
    Let $\varphi, \psi\in \cE^1$, $\mu\in \cM^1$, and let $\nu\doteq \MA_A(\tau)$ we then have:
    \[\left\lvert \int(\varphi-\psi)(\mathrm d\mu - \mathrm d\nu)\right\rvert \lesssim \Jd(\varphi, \psi)^q\Jd(\mu,\tau)^{\frac{1}{2}}R^{\frac{1}{2}-q}, \]
    for 
    \[q:= 2^{-n}, \quad \text{and} \quad R:=\max\{\Jd(\varphi), \Jd(\psi), \Ev(\mu), \Ev(\nu)\}.\]
\end{prop}
\begin{proof}
    If $\varphi, \psi\in \PL\cap PSH_A$ it follows from \cite[Theorem 2.23]{BJ23synthetic}. For general $\varphi, \psi\in \cE^1_A$ consider decreasing sequences $\varphi_i, \psi_i\in \PL\cap \PSH_A$ converging to $\varphi$ and $\psi$ respectively. 
    Then, by the monotone convergence theorem, the LHS of the inequality converges, while the RHS converges by the continuity of $\Jd$ along decreasing sequences.
\end{proof}

To conclude we will need the following result that can be found in the algebraic setting in \cite[Proposition 9.19]{BJ22trivval}. 
\begin{lemma}\label{lem:L1}
    Let $\varphi_i\in \PSH_A$ be a sequence of $A$-psh functions converging weakly to $\varphi\in \PSH_A$ and such that $\Jd_A(\varphi_i)$ is bounded. 
    Then for any $\mu$ in the image of the Monge–Ampère operator we have that
    \[\int \lvert \varphi_i-\varphi\rvert\, \mathrm d\mu \to 0.\]
\end{lemma}

\begin{proof}
The proof in \cite[Proposition 9.19]{BJ22trivval} applies directly to our setting, but we will give it here for the convenience of the reader.
    
Since $\varphi_i$ converge weakly to $\varphi$ we get that $\sup \varphi_i = \varphi_i(v_{\trivial})\to \varphi(v_{\trivial}) = \sup \varphi$. Hence we can suppose that $\sup \varphi_i=\sup \varphi=0$.

Let $\psi\in \cE^1_A$ be such that $\MA_A(\psi) = \mu$ and $\sup \psi = 0$, and let $\psi_j\in \cH_A$ be a sequence decreasing to $\psi$. Let us also denote by $\tilde\psi_j$ the sup-normalized functions $\psi_j - \sup \psi_j,$ and $\mu_j:=\MA_A(\psi_j)= \MA_A(\tilde\psi_j).$

We will start by proving that 
    \[\int (\varphi_i - \varphi)\, \mathrm d\mu \to 0.\]
Let $\epsilon>0$. Using the triangle inequality 
    \begin{align*}
        \left\lvert \int (\varphi_i - \varphi)\, \mathrm d\mu\right\rvert \le \left\lvert \int \varphi_i \, (\mathrm d\mu - \mathrm d\mu_j)\right\rvert + \left\lvert \int (\varphi_i - \varphi)\, \mathrm d\mu_j\right\rvert+ \left\lvert \int  \varphi\, (\mathrm d\mu- \mathrm d\mu_j)\right\rvert,
    \end{align*}
weak convergence and in the second inequality  Corollary~\ref{cor:MAestimate} we get that that for $i\gg 0$
        \begin{align*}
        \left\lvert \int (\varphi_i - \varphi)\, \mathrm d\mu\right\rvert &\le \left\lvert \int \varphi_i \, (\mathrm d\mu - \mathrm d\mu_j)\right\rvert + \epsilon + \left\lvert \int  \varphi\, (\mathrm d\mu- \mathrm d\mu_j)\right\rvert \lesssim\\
        &\lesssim \epsilon + 2 \Jd_A(\tilde \psi_j, \psi)^q\cdot \max\{C, \Jd_A(\psi), \Jd_A(\tilde \psi_j)\}^{1-q}=\\
        &= \epsilon +2 \Jd_A(\psi_j, \psi)^q\cdot \max\{C, \Jd_A(\psi), \Jd_A(\psi_j)\}^{1-q},
    \end{align*}
    where $q:= 2^{-n}$. Since the energy pairing is continuous along decreasing sequences, for $j$ sufficiently large we have
    \[\left\lvert \int (\varphi_i - \varphi)\, \mathrm d\mu\right\rvert\le 3\epsilon.\]
    Like in \cite[Proposition 9.19]{BJ22trivval}, we can then apply the same strategy to $\tilde \varphi_i:= \max\{\varphi_i, \varphi\}$, and observe that 
    \[\lvert \varphi-\varphi \rvert = 2 (\tilde\varphi_i - \varphi) + (\varphi_i - \varphi)\]
    to conclude.
\end{proof}

As a direct consequence we have the following result.
\begin{corollary}\label{cor:trivialestimate}
If $\varphi_i\in \cE^1_A$ converges strongly to $\varphi\in\cE^1_A$, then 
    \[\Jd_A(\varphi, \varphi_i)\to 0.\]
\end{corollary}

\addtocontents{toc}{\SkipTocEntry}
\subsection{$\cM^1$ as a quasi metric space}
Next we describe the \emph{quasi metric structure} on $\cM^1$ as developed by Boucksom–Jonsson \cite{BJ23synthetic} in the algebraic setting.

We start with a very general definition.
\begin{defi}
We say that a continuous map $\delta\colon M\times M \to \left[0,+\infty \right[$ is a \emph{quasi metric} for the topological space $M$ if there exists constants $C, D>0$ such that
\begin{align*}
            \delta(p,q) = 0 &\iff p =q\\
            C\cdot\delta(p, q)&\le\delta(p,t) + \delta(t, q) \\
            D^{-1}\cdot \delta(p,q)&\le \delta(q, p)\le D\cdot \delta(p,q),
        \end{align*}
        and if the topology generated by the sets $\{p\in M\mid \delta(p,q) < r\}$ coincides with the topology of $M$.
    \end{defi}

We now follow \cite[Theorems 2.23 and 2.25]{BJ23synthetic}.

\begin{theorem}\label{thm:quasimetric}
There exists an unique quasi metric on $\cM^1$, 
        \[\delta_A\colon \cM^1\times \cM^1\to \left [0, \infty\right[,\]
        such that:
        \begin{enumerate}
            \item  For all $\varphi\in \cE^1_A$ \[\delta_A(\mu, \MA_A(\varphi)) = \Jd_A(\mu, \varphi).\]
            \item For all $\varphi, \psi\in \cE^1$, $\mu, \nu\in \cM^1$
            \[\left\lvert \int(\varphi-\psi)(\mathrm d\mu - \mathrm d\nu)\right\rvert \lesssim \Jd_A(\varphi, \psi)^q\cdot\delta_A(\mu,\nu)^{\frac{1}{2}}\cdot R^{\frac{1}{2}-q}, \]
            for 
            \[q:= 2^{-n}, \quad \text{and} \quad R:=\max\{\Jd(\varphi), \Jd(\psi), \Ev(\mu), \Ev(\nu)\}.\]
            \item For all $\mu, \mu^\prime, \nu, \nu^\prime\in \cM^1$ 
            \[\left\lvert \delta_A(\mu, \nu) - \delta_A(\mu^\prime, \nu^\prime)\right\rvert \lesssim \max\{\delta_A(\mu, \mu^\prime), \delta_A(\nu, \nu^\prime)\}^q\cdot S^{1-q},\]
            for $S:= \max\{\Ev_A(\mu), \Ev_A(\mu^\prime), \Ev_A(\nu), \Ev_A(\nu^\prime)\}.$
        \end{enumerate}
    \end{theorem}
    \begin{proof}
        By \cite[Theorem 2.23]{BJ23synthetic} there exists an unique quasi metric $\delta_A$ as before that satisfies $\delta_A(\mu, \MA_A(\varphi)) = \Jd_A(\mu, \varphi)$ for all $\varphi\in \cH_A.$ 
        
        To get (1) it is then enough to observe that if $\varphi\in \cE^1_A$ and $\varphi_i\in \cH_A$ decreases to $\varphi$ then $\MA_A(\varphi_i)$ converges strongly to $\MA_A(\varphi)$ and $\Jd_A(\mu, \varphi_i)$ converges to $\Jd_A(\varphi)$ by Lemma~\ref{lem:MAcontdecreasing}  and Lemma~\ref{lem:Jmu} respectively.

        In order to prove (2) we observe that both the LHS and RHS of the inequality are continuous under decreasing limits, and then apply \cite[Theorem 2.23]{BJ23synthetic} for $\varphi$ and $\psi$ in $\cH_A.$

        The third point is a direct consequence of \cite[Theorem 2.22]{BJ23synthetic}.
    \end{proof}

As an important consequence we get the following result.
\begin{theorem}\label{thm:MAcont}
The Monge–Ampère operator $\MA_A\colon \cE^1_A \to \cM^1$ is continuous with respect to the strong topologies.
\end{theorem}

\begin{proof}
Let $\varphi_i\in \cE^1_A$ be a sequence that converges strongly to $\varphi\in \cE^1_A$.
Then 
\[\delta_A(\MA_A(\varphi_i), \MA_A(\varphi)) = \Jd_A(\varphi_i, \varphi)\to 0,\]
by Corollary~\ref{cor:trivialestimate}, and by Theorem~\ref{thm:quasimetric} we conclude that $\MA_A(\varphi_i)$ converges strongly to $\MA_A(\varphi).$
\end{proof}

\addtocontents{toc}{\SkipTocEntry}
\subsection{$\cE^1_A$ as a quasi metric space}

As we did for $\cM^1$, we will now give a quasi metric structure on $\cE^1_A$. In the algebraic setting this was done in \cite[Section 12.1]{BJ22trivval}.

 \begin{defi}
If $\varphi, \psi\in \cE^1_A$ we define  \[\partial_A(\varphi, \psi):= \Jd_A(\varphi, \psi) + \left\lvert \sup \varphi - \sup \psi \right \rvert\in \left[0, +\infty\right[. \] 
\end{defi}

Next we have the transcendental version of \cite[Theorem 12.4]{BJ22trivval}.
\begin{theorem}\label{thm:quasimetricE1}
The map 
\[\partial_A\colon \cE^1_A\times \cE^1_A\to \left[0, +\infty\right[\]
is a continuous quasi metric on $\cE^1_A$ with respect to its strong topology.
\end{theorem}
\begin{proof}
    By Corollary~\ref{cor:trivialestimate} the map $\partial_A$ is  continuous. Let us prove now that $\partial_A$ is a quasi metric.

    It is easy to see that $\partial_A$ is quasi symmetric, and the quasi triangle inequality follows from Lemma~\ref{lem:Jmu}. Hence, in order to prove that $\partial_A$ is a quasi metric it remains to check that $\partial_A(\varphi, \psi) = 0$ implies that $\varphi = \psi.$

Thus let $\varphi, \psi\in \cE^1_A$ be such that $\partial_A(\varphi, \psi) =0$. Since $A$-psh functions are completely determined on the set of divisorial valuations (cf. \cite[Theorem 4.3.8]{MP24transcendental}), it is enough to prove that \[\varphi|_{\Xdiv} = \psi|_{\Xdiv}.\]

Let $v\in \Xdiv$ be a divisorial valuation. By Example~\ref{ex:delta} the measure $\delta_v$ is of finite energy, which in turn implies that
\[0\le\left\lvert \varphi(v) - \psi(v)\right\rvert = \left\lvert \int(\varphi - \psi)\, (\delta_v - \delta_{v_{\trivial}})\right\rvert\lesssim 0\]
by Theorem~\ref{thm:quasimetric}.

We now want to prove that if $\varphi_i, \varphi\in \cE^1_A$ are such that \(\partial_A(\varphi_i, \varphi)\to 0 \), then \( \varphi_i\overset{s}{\longrightarrow}\varphi.\)
For that, first observe that if $\partial_A(\varphi_i, \varphi)\to 0$ then 
    \begin{enumerate}
        \item $\lvert \sup \varphi_i - \sup \varphi\rvert \to 0;$
        \item $\Jd_A(\varphi_i, \varphi)\to 0$, in particular $\Jd_A(\varphi_i)$ is a bounded sequence by Lemma~\ref{lem:Jmu}.
    \end{enumerate}
Let $v\in \Xdiv$ be a divisorial valuation. As before we estimate
    \begin{align*}
        \left\lvert \varphi_i(v) - \varphi(v)\right\rvert &\le \lvert \sup \varphi_i - \sup \varphi\rvert + \left\lvert \int(\varphi_i - \varphi)\, (\delta_v - \delta_{v_{\trivial}})\right\rvert\\
        &\lesssim \lvert \sup \varphi_i - \sup \varphi\rvert + C \cdot \Jd_A(\varphi_i, \varphi)^q \to 0,
    \end{align*}
    where $q:= 2^{-n}$, and $C>0$ is a constant independent of $i$. This shows that $\varphi_i$ converges weakly to $\varphi$.

To conclude, Lemma~\ref{lem:L1} gives the convergence of the integrals $\int\varphi_i\,\MA_A(\varphi)$ to the integral $\int \varphi\, \MA_A(\varphi),$ and thus 
\[\lvert \E_A(\varphi_i) - \E_A(\varphi)\rvert \le \Jd_A(\varphi, \varphi_i) + \left\lvert \int(\varphi_i - \varphi)\,\MA_A(\varphi)\right\rvert \to 0,\]
and we are done.
\end{proof}

\section{The non-Archimedean Calabi–Yau theorem}

The goal of this section is to prove Theorem \ref{thm:B} which says that the Monge–Ampère operator is a homeomorphism between $\cE^1_{A,\sup}$ and $\cM^1$.

We start by establishing injectivity.

\begin{prop}\label{prop:injective}
The Monge-Amp\`ere operator $\MA_A: \cE^1_{A,\sup}\to \cM^1$ is injective.
\end{prop}
\begin{proof}
We follow the proof in the algebraic setting \cite[Corollary 10.4]{BJ22trivval}. 
    
If $\varphi,\psi\in\cE^1_{A,\sup}$ are such that $\MA(\varphi) = \MA(\psi)$ we get
\[\mathrm J(\varphi, \psi) +\mathrm J(\psi, \varphi) = \int (\varphi - \psi)\left(\MA(\psi) - \MA(\varphi)\right) = 0,\]
which by the non-negativity of the Dirichlet functional implies that $\mathrm J(\varphi, \psi) = 0 = \mathrm J(\psi, \varphi).$
This implies that $\partial_A(\varphi, \psi) = 0$, and by Theorem~\ref{thm:quasimetricE1} we finally get that $\varphi = \psi.$
\end{proof}

Following \cite{BJ23synthetic} we introduce a key notion for proving surjectiveness, namely that of a maximizing sequence.

\begin{defi}
A sequence $\varphi_i\in \cE^1_A$ is a \emph{maximizing sequence} for $\mu\in \cM^1$ if \[\E_A(\varphi_i) - \int \varphi\, \mathrm d\mu \to \Ev_A(\mu). \]
\end{defi}

\begin{prop}\label{prop:MAdense}
If $\varphi_j\in \cE^1_A$ is a maximizing sequence for $\mu\in \cM^1$, then $\MA_A(\varphi_j)$ converges strongly to $\mu$. In particular, if \[\mathrm E_A^\vee(\mu) = \mathrm E_A(\varphi) - \int\varphi\, \mathrm d\mu,\]
    then $\MA(\varphi) = \mu.$
\end{prop}
\begin{proof}
Since 
\[\delta_A(\mu, \MA_A(\varphi_j)) = \Ev_A(\mu) - \Jd_A(\mu, \varphi_j)\to 0\] the first statement follows from Theorem~\ref{thm:quasimetric}. The second statement then follows from the observation that $\varphi_i=\varphi$ is a maximizing sequence for $\mu$.
\end{proof}

We are now ready to prove the surjectivity of the Monge-Amp\`ere operator.
\begin{prop} \label{thm:surj}
The Monge–Ampère operator $\MA_A\colon \cE^1_{A,\sup}\to \cM^1$ is surjective.
\end{prop}
\begin{proof}
We follow the strategy of \cite[Theorem 12.8]{BJ22trivval}.
    
Let $\mu\in \cM^1$ and let $\varphi_i\in \cH_A$ be a sup-normalized maximizing sequence for $\mu$. By Theorem~\ref{thm:compactness} $\PSH_{A,\sup}$ is weakly compact. Thus, after possibly taking a subsequence, we can suppose that $\varphi_i$ converges weakly to some $\varphi \in \PSH_{A,\sup}$.
Our goal is to prove that $\varphi \in \cE^1_A$ and that moreover $\MA_A(\varphi) = \mu$.

By Lemma~\ref{lem:Jmu} $\Jd_A(\varphi_i)$ is bounded which, by lower semicontinuity of $\mathrm J$ in the weak topology, implies that $\Jd_A(\varphi)$ is finite. This in turn implies that $\varphi \in \cE_A^1$.
    
We now claim that $\int \varphi_i\,\mathrm d\mu \to \int\varphi\,\mathrm d\mu.$ 
Knowing this, by the weak upper semicontinuity of the energy functional $\E_A$ we would have that 
    \[\Ev_A(\mu)\ge \E_A(\varphi) - \int \varphi \,\mathrm d\mu \ge \limsup \left\{\E_A(\varphi_i) - \int \varphi_i \,\mathrm d\mu \right\}= \Ev_A (\mu),\]
which by Proposition~\ref{prop:MAdense} would show that $MA_A(\phi)=\mu$. 

Therefore, let us prove the claim. Let $\psi_j$ be any maximizing sequence for $\mu$, and $\mu_j:=MA_A(\psi_j)$. Let $R\gg0$ be sufficiently large so that \[\Jd_A(\varphi_i)<R, \text{ and } \Ev(\mu)<R.\]
In particular this implies that $\Jd_A(\varphi)\le R$ and $\Ev_A(\mu_j)\le R$ for sufficiently large $j$.
Applying Theorem~\ref{thm:quasimetric} we then have that
     \begin{align*}
        \left\lvert \int (\varphi-\varphi_i) \,\mathrm d \mu \right\rvert &\le \left\lvert \int (\varphi- \varphi_i) (\,\mathrm d \mu - \,\mathrm d\mu_j)\right\rvert +\left\lvert \int (\varphi_i-\varphi) \,\mathrm d \mu_j\right\rvert \lesssim \\
        &\lesssim \Jd_A(\mu, \mu_j)^q \cdot R^{1-q} + \left\lvert \int (\varphi_i-\varphi) \,\mathrm d \mu_j\right\rvert,
    \end{align*}
    where $q:= 2^{-n}.$
    For $j$ sufficiently big we then have that
    \[\left\lvert \int (\varphi-\varphi_i) \,\mathrm d \mu \right\rvert \lesssim \epsilon + \left\lvert \int (\varphi_i-\varphi) \,\mathrm d \mu_j\right\rvert,\]
    which by Lemma~\ref{lem:L1} tends to $\epsilon$ as $i$ tends to $\infty$. As $\epsilon>0$ was arbitrary we are done.
\end{proof}

We now prove Theorem~\ref{thm:B}:

\begin{proof}[Proof of Theorem \ref{thm:B}]
By Proposition \ref{prop:injective} and Theorem~\ref{thm:surj} the  Monge–Ampère operator is a bijection between $\cE^1_{\sup, A}$ and $\cM^1$. 

Combining the first points of Lemma~\ref{lem:Jmu}  and Theorem~\ref{thm:quasimetric}  we obtain 
\[\delta_A(\MA_A(\varphi), \MA_A(\psi)) = \Jd_A(\varphi, \psi) = \partial_A(\varphi, \psi),\]
where the second equality is given by the sup normalization.
Therefore, the Monge--Ampère operator is a quasi isometry between $\cE^1_{A,\sup}$ and $\cM^1$, and in particular, a homeomorphism. 
\end{proof}

\section{Measures supported on dual complexes}

Let $\Delta_{\cX}$ be the dual complex associated to a SNC model $\cX$ as defined in the introduction. Let $\Delta$ denote the projective limit of the collection $\Delta_{\cX}$, and let $p_\cX\colon \Delta\to \Delta_\cX$ denote the projection map.

It is easy to see that there is a linear isomorphism between the real valued functions that are affine in each face $\sigma\in \Delta_\cX$ and the set of vertical divisors $\VCar(\cX)$, and thus 
\begin{equation}\label{eq:PLiso}
    \PL(X^{\NA}) \simeq \varinjlim_{\cX} \VCar(\cX) \simeq \bigcup_{\cX} p_\cX ^*(\Aff(\Delta_\cX)).
\end{equation}
We will let $\PL(\Delta):=\bigcup_{\cX} p_\cX ^*(\Aff(\Delta_\cX))$.

In \cite[Theorem 3.1.1]{MP24transcendental} it is proved that by taking the tropical spectrum\footnote{We can endow both $\PL(X^{\NA})$ and $\PL(\Delta)$ two operations that make it a semiring, cf. \cite[Appendix A]{MP24transcendental}.} of $\PL(X^{\NA})$ and $\PL(\Delta)$, by a tropical version of the Gelfand representation theorem we obtain an homeomorphism
\[X^{\NA}\simeq \tspec \PL(X^{\NA})\overset{p}{\longrightarrow} \tspec \PL(\Delta)\simeq \Delta,\]
 that preserves the PL structures.

 Moreover, using \emph{monomial valuations} 
 %\color{blue} why not denote the map by $i_{\cX}$? \color{black} 
 we can construct a map 
 \[i_\cX\colon \Delta_{\cX}\hookrightarrow X^{\NA},\]
 that maps each vertex $e_i\in \Delta_{\cX}$ --corresponding to the irreducible component $E_i\subseteq \cX_0$-- to the divisorial valuation $v_{E_i}\in X^{\NA}$.
 For more details see \cite[Section 3.1.1]{MP24transcendental}.

The goal of this Section is to prove the for us important result that any $\mu\in \cM^1$ can be well approximated by finite energy measures, each of which is supported on some dual complex (see Corollary \ref{cor:energyentropy}).

\addtocontents{toc}{\SkipTocEntry}
\subsection{$A$-psh functions and dual complexes}

In this section we will prove that if $\varphi\in \PSH_A$, then the restriction of $\varphi$ to a dual complex $\Delta_{\cX}\hookrightarrow X^{\NA}$ is continuous and convex. 
The arguments in this section mostly follow \cite{BFJ16semipositive}.

Let us start with a simpler statement: 
\begin{lemma}\label{lem:pshinequality}
    Let $\varphi\in \PSH_A$, and let $\cX$ be a SNC model. Then we have that
    \[\varphi \le \varphi\circ p_\cX,\]
where $p_\cX\colon X^{\NA} \to \Delta_{\cX}$ denotes the retraction of $X^{\NA}$ onto the dual complex $\Delta_\cX$.
\end{lemma}
The original version of this Lemma in the algebraic setting can be found in \cite[Proposition 5.9]{BFJ16semipositive}. 
The proof presented here uses the description of monomial valuations as Lelong numbers, as in \cite[Appendix B]{MP24transcendental}.
\begin{proof}[Proof of Lemma~\ref{lem:pshinequality}]
It is enough to consider the case when $\varphi\in \cH_A$, as the general case then follows from taking a decreasing limit.
Let $\varphi=\varphi_D$ and $\cX\overset{\mu}{\longrightarrow} X\times \Pro$ a model on which $D$ lives. Let also $U=U_{\le \varphi}$ be the  maximal geodesic ray associated to $\varphi$.
    
Since both $\varphi$ and $\varphi\circ p_\cX$ are continuous,  it is enough to check that $\varphi(v)\le \varphi\circ p_\cX(v)$ for every $ v\in \Xdiv$.

Let $v=v_{E^{\prime}}$ be a divisorial valuation corresponding to an irreducible component $E^\prime$ of the central fiber of a SNC model $\cX^\prime$ dominating $\cX$.
    
By \cite[Appendix B]{MP24transcendental} it follows that given the morphisms:

        \begin{center}
            % https://tikzcd.yichuanshen.de/#N4Igdg9gJgpgziAXAbVABwnAlgFyxMJZABgBpiBdUkANwEMAbAVxiRAB12BjADQD1OaAE5YAtqwC+pdJlz5CKAIylFVWoxZtOvEFJnY8BImQBMa+s1aIQPTnnFxOABSERdamFADm8IqABmrqJIZCA4bogm1Baa1pyiTALswmKS0iCBEMGIymERUeqWWuwJuumZ2aHhSLkxVhzsQgAWbhIUEkA
            \begin{tikzcd}
                \cX^\prime \arrow[dd, "\mu^\prime"] \arrow[rd, "\rho"] &                       \\
                                                       & \cX \arrow[ld, "\mu"] \\
                X\times\Pro                                            &                      
            \end{tikzcd}
        \end{center}
    and $w = (v(E_1), \dotsc, v(E_k))$ for $E_1, \dotsc, E_k$ the irreducible components of $\cX_0$ we have
    \[\varphi(v) = -\nu_{E^{\prime}}(U\circ\mu^\prime), \quad \text{and}\quad  \varphi(p_\cX(v)) = -\nu_w(U\circ \mu, p),\]
    for any generic choice of $p\in Z(v, \cX)$.
    Then the matter reduces to a question on Lelong numbers which is easy to check.

    Indeed, let $q\in E^\prime$ be any point not contained in any other irreducible component of $\cX^\prime_0$, and $p:= \rho(q)$. 
    The claim then reduces to showing that
    \[U\circ \mu \circ \rho \le \nu_w \log \lvert z^\prime\rvert + O(1), \quad\text{ for  }\quad \nu_w:= \nu_w(U\circ \mu, p),\]
    and $z^\prime$ a local equation of $E^\prime$ around $q$.
    Let $z_1, \dotsc, z_n$ be local coordinates around $p$ such that $E_1, \dotsc, E_k$ are locally given by $z_1, \dotsc, z_k$.
    Since for every $i$ we have $U\circ \mu \le \nu_w \frac{\log\lvert z_i\rvert}{w_i}$ precomposing with $\rho$ we get
    \[U\circ \mu \circ \rho \le \nu_w \frac{\log\lvert (z^\prime)^{w_i}\rvert}{w_i}+ O(1) = \nu_w \log \lvert z^\prime\rvert + O(1).\]

\end{proof}

Recall that given a model $\cX$, $i_{\cX}\colon \Delta_\cX \to X^{\NA}$ is the map that identifies the dual complex $\Delta_\cX$ with the corresponding monomial valuations in $X^{\NA}$.

Let us state now the main theorem of this section, namely the transcendental analogue of \cite[Proposition 7.5]{BFJ16semipositive}.
%The goal of this section is to prove the following theorem:
\begin{theorem}\label{thm:pshcont}
Given $\varphi \in \PSH_A$ and a model $\cX$, then the restriction $\varphi\circ i_{\cX} \colon \Delta_\cX\to \R$ is continuous and convex. 
\end{theorem} 

Our proof will adapt the results of both \cite[Proposition 7.5]{BFJ16semipositive} and \cite[Section 11]{BJ22trivval}.

We start by proving the following simpler statement:
\begin{lemma}\label{lem:hconvex}
If $\varphi\in \cH_A$, then $\varphi\circ i_{\cX}$ is convex.
\end{lemma}
\begin{proof}
Let $D\subseteq \cX^\prime\overset{\mu}{\longrightarrow}X\times\Pro$ be a vertical divisor defining $\varphi$. We say that $\ia$ is a \emph{flag ideal} if it is a $\C^*$-invariant coherent ideal sheaf of $\ia \subseteq \cO_{X\times\Pro}$ that is supported on $X\times\{0\}.$ 
Since $\varphi\in \cH_A$ is a non-Archimedean Kähler potential, $\cO_{\cX^\prime}(D)$ is $\mu$-globally generated, and thus we can find a flag ideal $\ia$ of $X\times\Pro,$ such that 
     \[\ia\cdot \cO_{\cX^\prime} = \cO_{\cX^\prime}(D).\]
    
    We are left to check that, for every face $\sigma_Z\subseteq \cX$, the map:
    \[\sigma_Z\ni w\mapsto \varphi(i_\cX(w)) = -i_\cX(w)(\ia)\]
    is convex.
    This follows now, by the definition of $i_\cX(w)$, cf. \cite[Proposition 2.3.4]{MP24transcendental}, if $p\in Z$ and $f_1, \dots, f_\ell$ are local generators of $\ia_p$, then
    \[-i_\cX(w)(\ia) = \max_{1\le j\le \ell} - i_\cX(w)(f_j).\]
    It follows from \cite[Equation 2.3.2]{MP24transcendental} that $- i_\cX(w)(f_j)$ is a convex piecewise linear function, which concludes the proof.
\end{proof}

Since every $A$-psh function is the pointwise limit of functions in $\cH_A$, its restriction to a dual complex is convex and therefore continuous on the interior of each face.

To establish continuity at the vertices, we will the adapt the following result from \cite[Theorem 11.12]{BJ22trivval}. 

\begin{prop}\label{prop:contbound}
    Let $K\subseteq \Xnp$ be a bounded set for the $d_\infty$ distance, then for every $\varphi$ $A$-psh function we have that the restriction $\varphi_{|K}$ to $K$ is continuous.
\end{prop}

\begin{proof}
Recall from Section \ref{sec:np} that $\Xnp$ denotes the set on non-pluripolar points, and that $$d_{\infty}(v,w):=\sup_{\varphi\in \PSH} \lvert \varphi(v) - \varphi(w)\rvert.$$

The proof follows that of \cite[Theorem 11.12]{BJ22trivval} and is added here for convenience of the reader.

Let us first observe that:
    \begin{itemize}
        \item Since $\exp\colon \left]-\infty, 0\right]\to\R$ is convex with derivative between zero and one, we have that $\psi:=\exp(\varphi - \sup \varphi)$ is a bounded $A$-psh function, and thus is a potential of finite energy.
        \item $\varphi_{|K}$ is continuous if and only if $\psi_{|K}$ is continuous.
    \end{itemize}
    Hence it enough to prove the result for finite energy potentials.
    Let $\psi_i\in \cH$ be a decreasing sequence to $\psi$, we then have:
    \[0\le\Jd(\psi_i, \psi)= \E(\psi) - \E(\psi_i) + \int (\psi_i - \psi)\MA(\psi)\le  \int (\psi_i - \psi)\MA(\psi) \to 0,\]
    where the second inequality is given by $\psi_i\ge \psi$, and the limit by the monotone convergence theorem. 
    In particular, $\Jd(\psi_i)$ is uniformly bounded.
    
    Let $C>0$ be big enough such that $\Jd(\psi_i)\le C$, and that $K\subseteq \{v\in X^{\NA} \mid T(v)\le C\}$. Then, for every $v\in K$
    \begin{align*}
        \lvert \psi_i(v) - \psi(v)\rvert &\le \lvert \psi_i(v) - \psi_i(v_{\trivial}) - \psi(v) + \psi(v_{\trivial})\rvert + \lvert \psi_i(v_{\trivial}) - \psi(v_{\trivial})\rvert\\
        &=\left\lvert \int(\psi_i -\psi)(\delta_{v} - \delta_{v_{\trivial}})\right\rvert + \lvert \psi_i(v_{\trivial}) - \psi(v_{\trivial})\rvert =(\star),
    \end{align*}
    which, by Proposition~\ref{prop:2.23} together with Proposition~\ref{prop:Tinequality}, implies
    \begin{equation}\label{eq:uniform}
        (\star)\le \Jd(\psi_i, \psi)^q\cdot C^{1-q} + \lvert \psi_i(v_{\trivial}) - \psi(v_{\trivial})\rvert \to 0,
    \end{equation}
    where $q:=2^{-n}$. 
    Since the RHS of \eqref{eq:uniform} is uniform on $v$, this implies that ${\psi_i}_{|K}$ converges uniformly to $\psi_{|K}$, which in turn implies that $\psi_{|K}$ is continuous. 
\end{proof}

We are now ready to prove Theorem~\ref{thm:pshcont}.
\begin{proof}[Proof of Theorem~\ref{thm:pshcont}]
Let us start proving that $\varphi_{|\Delta_\cX}$ is convex. Let $\varphi_i\in \cH_A$ be a decreasing sequence converging to $\varphi$.
Then, by Lemma~\ref{lem:hconvex} $\varphi_{|\Delta_\cX}$ is convex.
Since convexity is preserved under pointwise limits, we then have that the restriction $\varphi_{|\Delta_\cX}$ is convex.

Now, the restriction $\varphi_{|\overset{\circ}{\sigma_Z}}$ to the interior of any face $\sigma_Z\subseteq \Delta_{\cX}$ is continuous by convexity.
Therefore, it is enough to prove that $\varphi$ is continuous on every vertex $v\in \Delta_\cX$.
Let $\Sigma:=\{v_1, \dots v_\ell\}$ be the set of vertices of $\Delta_{\cX}$, the divisorial valuations attached to irreducible components of $\cX$.
By Example~\ref{ex:bounded} the set $\Sigma$ is bounded in $d_\infty$, and hence we conclude by applying Proposition~\ref{prop:contbound}.
\end{proof}

The next result will be useful in Section~\ref{sec:divisorialmeasures}.
\begin{lemma}\label{lem:P=sup}
    Let $\{v_1, \dotsc , v_\ell\}\subseteq \Xdiv$ be a set of divisorial points, and $t\in \R^\ell$.
    Then there exists a PL function $f\in \PL$ such that
    \[\Pe_A(f) = \sup\{\varphi\in \PSH_A\mid \varphi(v_i)\le t_i\}.\] 
Furthermore, if $\DcX\hookrightarrow X^{\NA}$ is a dual complex containing the set $\{v_1, \dotsc, v_\ell\}$ as vertices, then $f$ can be taken to be $\varphi_D$, for the vertical divisor $D = \sum_{i} t_i\,  E_i$, where $v_{E_i} = v_i$.
\end{lemma}
\begin{proof}
We follow the proof in the algebraic setting \cite[Lemma 8.5]{BFJ15solution}.
    It is clear that if $\varphi\in \PSH_A$ satisfies $\varphi \le f = \varphi_D$ as above, then $\varphi(v_i)\le t_i$.
    
    Moreover, if $\varphi\circ p_\cX (v_i) =\varphi(v_i)\le t_i$ by convexity of $\varphi\circ p_\cX$, we have \[\varphi \circ p_\cX \le f\circ p_\cX = f,\] which by Lemma~\ref{lem:pshinequality} implies $\varphi \le f$, concluding the proof.
\end{proof}

\addtocontents{toc}{\SkipTocEntry}
\subsection{Smoothing measures}
Let $\mu\in \cM^1$ be a measure of finite energy. For each SNC model $\cX$ we define the pushforward measure $\mu_{\cX}:=(p_\cX)_*\mu$.
The measure $\mu_\cX$ is a measure on $\Delta_\cX\hookrightarrow X^{\NA}$ that we can identify as a subset of $X^{\NA}$ using monomial valuations as in \cite{MP24transcendental}.
Hence, we see $\mu_\cX$ as a probability measure on $X^{\NA}$ supported on the dual complex $\Delta_\cX$.

With the dual complex description of $X^{\NA}$ it is easy to see that $\mu_\cX$ converges weakly to $\mu$.

By the next result, adapted from a preprint version of \cite{BJ22trivval}, we see that $\mu_\cX$ actually converges strongly to $\mu$.
\begin{lemma}\label{lem:increasingenergy}
If $\mu\in \cM^1$, then the net $\Ev(\mu_\cX)$ is eventually increasing, and moreover 
    \[\lim_{\cX} \Ev(\mu_\cX) = \Ev (\mu).\]
\end{lemma}
\begin{proof}
    Recall that the energy of $\nu\in \cM^1$ is given by $\Ev(\nu) = \sup\{\E(\varphi) - \int\varphi\,\mathrm d \nu\mid \varphi \in \cE^1\}.$

    By Lemma~\ref{lem:pshinequality}, we have that $(\varphi\circ p_\cX)_\cX$ is a decreasing net that converges to $\varphi$, cf. \cite[Section 2]{MP24transcendental}.
    Therefore, the integral  \[\int\varphi\,\mathrm d\mu_\cX = \int (\varphi\circ p_\cX)\, \mathrm d\mu\] is decreasing as well, and by the monotone convergence theorem the result follows.
\end{proof}

We recall that the non–Archimedean entropy functional is defined as:
\[\Ent(\mu) := \sup_\cX \int (A_X\circ p_\cX)\mathrm d\mu,\]
where $A_X$ denotes the log discrepancy function.
\begin{lemma}
For $\mu\in \cM^1$ we have that 
    \[\sup_{\cX} \Ent(\mu_\cX) = \Ent(\mu).\]
\end{lemma}
\begin{proof}
    This follows directly from the definition of the entropy, together with the identity $\int(A\circ p_\cX)\,\mathrm d\mu = \int A\,\mathrm d\mu_\cX.$
\end{proof}

\begin{corollary}\label{cor:energyentropy}
For any $\mu\in \cM^1$ we can find a sequence $\mu_j\in\cM^1$ supported on dual complexes $\Delta_{\cX^j}\hookrightarrow X^{\NA}$, converging strongly to $\mu$ and such that
    \[\Ent(\mu_{j}) \to \Ent(\mu).\]
\end{corollary}
\begin{proof}
By the previous lemma we can pick a sequence of models $\cX^j$, such that $\mu_j:=\mu_{\cX_j}$ converges weakly to $\mu$, and such that \[\lim_{j\to \infty}\Ent(\mu_j) = \Ent(\mu).\]
Since the energy $\Ev(\mu_\cX)$ is increasing by Lemma~\ref{lem:increasingenergy}, we then directly have that $\Ev(\mu_j)\to \Ev(\mu).$
\end{proof}

Given $\mu\in \cM^1$ and a model $\cX$ the entropy $\Ent(\mu_{j, \cX})$ is defined to as \[\Ent(\mu_{\cX}):=\int A\, \mathrm d\mu_{\cX} = \int (A\circ p_\cX)\, \mathrm d\mu_j.\]
We will later have use for the following lemma.
\begin{lemma}\label{lem:weakentropy}
    Let $\mu_j, \mu\in \cM^1$ be measures of finite energy such that $\mu_j \rightharpoonup \mu$ weakly.
    Then, for every model $\cX$  \[\Ent(\mu_{j, \cX})\to \Ent(\mu_\cX).\]
\end{lemma}
\begin{proof}
    
    Since $A\circ p_\cX$ is a PL function, in particular continuous, by weak convergence we have
    \[\int (A\circ p_\cX) \mathrm d\mu_j\to \int (A\circ p_\cX) \mathrm d\mu = \int A\, \mathrm d\mu_\cX\]
    that coincides with $\Ent(\mu_\cX).$
\end{proof}

\section{Regularity of solutions}
In this section we will prove an analogue of \cite[Theorem 12.12]{BJ22trivval} which says that if $\mu\in \cM^1$ is a measure of finite energy supported on a dual complex $\Delta_\cX$, then the solution of $\MA_A(\varphi) = \mu$ is continuous. 

The proof follows closely that in the algebraic setting, and relies on Theorem~\ref{thm:pshcont} together with a comparison principle, Theorem~\ref{thm:comparison}.

\addtocontents{toc}{\SkipTocEntry}
\subsection{The comparison principle}

\begin{theorem}\label{thm:comparison}
Given $\varphi, \psi\in \cE^1_A$ we let $U:=\{x\in X^{\NA}\mid \varphi(x)>\psi(x)\}$. We then have that:
    \[\ind_U\MA(\max\{\varphi, \psi\}) = \ind_U\MA(\varphi).\]
\end{theorem}

\begin{proof}
This is an analogue of \cite[Theorem 7.40]{BJ22trivval} and the proof is exactly the same. It is presented here for the convenience of the reader.
Let us first suppose that $\varphi, \psi \in \PL\cap \PSH_A$. Let $\cX$ be a SNC model with $D_1, D_2, G\in \VCar(\cX)$ such that:
    \[\varphi = \varphi_{D_1}, \quad \psi = \varphi_{D_2}, \quad \max\{\varphi, \psi\} = \varphi_G.\]
    
Since we are dealing PL functions their Monge–Ampère measures are explicit. Let $\cX_0 = \sum b_i E_i$ be the irreducible decomposition of the central fiber. We need to show that if $v_{E_i}\in U$, i.e. if $\varphi(v_{E_i})>\psi(v_{E_i})$, then $(A + D_1)^n\cdot E_i = (A + G)^n\cdot E_i$.
To do that we will show that $D_1$ and $G$ coincide in a neighborhood of $E_i$.

Indeed, we will show that if $v_{E_i}\in U$, then for every $E_j$ that intersects $E_i$, we have $\varphi(v_{E_j}) \ge \psi(v_{E_j})$.
This turns out to be a general fact about piecewise linear functions, that we will prove only in this specific case.
Observe that $G-D_1\ge 0$ is an effective divisor, and since $v_{E_i}\in U$ we have
    \begin{align*}
        \varphi(v_{E_i}) = \max\{\varphi, \psi\}(v_{E_i}) &\implies \max\{\varphi, \psi\}(v_{E_i}) - \psi(v_{E_i})>0\\
        &\implies \ord_E(G- D_2)>0
    \end{align*}
    in particular $E$ is the support of $G-D_2$.
    Hence, for every divisorial valuation $v$ centered in $\cX$ on $E_i$, we will have:
    \[\max\{\varphi, \psi\}(v)>\psi(v)\implies \varphi(v) = \max\{\varphi, \psi\}(v).\]
    Since $\cX_0$ is SNC and each $E_i$ is $\C^*$-invariant, the intersection $Z:= E_i\cap E_j$ is a $\C^*$-invariant submanifold. Thus, by blowing up $Z$ we obtain a new model $\cX^\prime$.
    If we denote by $v$ the divisorial valuation associated with the strict transform of $Z$ on $\cX^\prime$, we then have that the center of $v$ on $\cX$ \[Z(v, \cX) = Z.\] 
    Hence, $\varphi(v) = \max\{\varphi, \psi\}(v)$ as we discussed above.
    This implies that $G$, and $D_1$ coincide in $Z$, which in turn implies that $G$ and $D_1$ coincide on $E_j$.
    
For the general case, consider decreasing sequences $\varphi_j$ and $\psi_j$ in $\PL\cap \PSH_A$ that converge to $\varphi$ and $\psi$ respectively. Then we follow this scheme:
    \begin{itemize}
        \item We can replace the indicator function in the statement for $f= \max\{\varphi -\psi, 0\}$
        \item Applying the argument above we then have
        \[f_j\MA(\max\{\varphi_j, \psi_j\}) = f_j \MA(\varphi_j)\]
        for $f_j = \max\{\varphi_j - \psi_j, 0\}$.
        \item Since the Monge–Ampère operator is continuous along decreasing sequences, we have $\MA(\max\{\varphi_j, \psi_j\}) \to \MA(\max\{\varphi, \psi\})$, that together with $f_j\to f$ implies the result.
    \end{itemize}
\end{proof}

\begin{corollary}\label{cor:comparison}
    Let $\varphi, \psi \in \cE^1$, and $K:=\supp \MA(\varphi)$.
    If \[\varphi_{|K}\ge \psi_{|K},\] we then have $\varphi\ge \psi.$
\end{corollary}
\begin{proof}
Let $\epsilon>0$ be small, and let $\varphi_\epsilon$ denote $\varphi +\epsilon$, and $U_\epsilon:= \{\varphi_\epsilon>\psi\}$. Observe that $U_\epsilon \supseteq K$.
    By Theorem~\ref{thm:comparison}, we have that
    \begin{align*}
    \ind_{U_\epsilon}\MA(\max\{\varphi_\epsilon, \psi\}) = \ind_{U_\epsilon} \MA(\varphi_\epsilon) &=\ind_{U_\epsilon}\MA(\varphi)\\
    &= \MA(\varphi),
    \end{align*}
    which implies $\MA(\max\{\varphi_\epsilon, \psi\}) = \MA(\varphi).$
    By Proposition~\ref{prop:injective} we then get that there exists a constant $c_\epsilon$ such that \[\varphi + c_\epsilon = \max\{\varphi_\epsilon, \psi\} = \max\{\varphi, \psi - \epsilon\} + \epsilon,\]
    since $\varphi = \max\{\varphi, \psi-\epsilon\}$ at some point –for instance in any point of the support of $\MA(\varphi)$– we get that $c_\epsilon = \epsilon$, and the result follows.
\end{proof}

The next result is an analogue of the ``easy direction" of \cite[Theorem 8.10]{BJ24green}.
\begin{corollary}\label{cor:continuousenevlope}
If $\varphi\in \CPSH_A$ is a continuous psh function such that $\MA_A(\varphi)$ is supported on a dual complex $\DcX\hookrightarrow X^{\NA}$, then
    \[\varphi = P_A(\varphi\circ p_\cX).\]
\end{corollary}
\begin{proof}
 The proof is exactly the same as in \cite[Theorem 8.10]{BJ24green} and is given here for completeness.
    By Lemma~\ref{lem:pshinequality} $\varphi \le \varphi \circ p_\cX$, hence we have $\varphi \le P_A(\varphi\circ p_\cX)$.

    On the other hand, on $\DcX\hookrightarrow X^{\NA}$ the restriction $\varphi |_{\DcX}$ coincides with $\Pe(\varphi\circ p_\cX)|_{\DcX}$.
    In particular, $\Pe(\varphi\circ p_\cX)\le \varphi$ in the support of $\MA_A(\varphi)$.
    By Corollary~\ref{cor:comparison} we get $\Pe(\varphi\circ p_\cX)\le \varphi$ everywhere, and the result follows.
\end{proof}

\addtocontents{toc}{\SkipTocEntry}
\subsection{Continuity of solutions}

\begin{theorem}\label{thm:continuity}
If the support of $\mu\in \cM^1$ is contained in the dual complex $\Delta_\cX$ of a model $\cX$, then the solution $\varphi$ to the Monge-Amp\`ere equation $\MA_A(\varphi) = \mu$ is continuous.
\end{theorem}
\begin{proof}
Let $\varphi\in \cE^1$ be say the sup-normalized solution to $\MA(\varphi) = \mu$, and let $K:=\supp \mu\subseteq \Delta_{\cX}$.

Consider now a decreasing sequence $\varphi_k\in \cH_A$ converging to $\varphi$.
By Theorem~\ref{thm:pshcont} we have that $\varphi_{|\Delta_{\cX}}$ is continuous, hence by Dini's theorem ${\varphi_k}_{|\Delta_\cX}$ converges uniformly to $\varphi_{|\Delta_{\cX}}$.
Therefore, for every $\epsilon>0$ if $k\gg0$ is sufficiently big we have:
    \[\varphi_{|\Delta_{\cX}} \le {\varphi_k}_{|\Delta_{\cX}} \le \epsilon +\varphi_{|\Delta_{\cX}},\]
    in particular ${\varphi_k}_{|K} \le \epsilon +\varphi_{|K}$.

By the comparison principle of Corollary~\ref{cor:comparison}, we have:
\[\varphi\le \varphi_k \le \epsilon +\varphi,\]
    which implies that $\varphi_k$ converges uniformly to $\varphi$ on $X^{\NA}$, and thus $\varphi$ is continous.
\end{proof}

\addtocontents{toc}{\SkipTocEntry}
\subsection{Divisorial measures and envelopes}
\label{sec:divisorialmeasures}
\begin{defi}
A probability measure on $X^{\NA}$ of the form $\mu=\sum_{i=1}^N a_i\delta_{v_i}$, where each $v_i\in \Xdiv$, is called a \emph{divisorial measure}, and the set of divisorial measures is denoted by $\cM^{\divisorial}$.
\end{defi}

Note that $\cM^{\divisorial}\subseteq \cM^1$.

The purpose of this section is to provide a description of the solution of the Monge-Ampère equation
\[\MA(\varphi) = \mu,\]
when $\mu\in \cM^{\divisorial}$.

The next result is a transcendental analogue of \cite[Proposition 8.6]{BFJ15solution}.
\begin{prop}\label{prop:divisorialenvelope}
If $\mu\in \Md$, then there exists a PL function $f$ such that:
    \[\MA_A(P_A(f)) = \mu.\]
\end{prop}
\begin{proof}
    Let $\cX$ be a SNC model such that the support of $\mu$ is contained in the set of vertices $\{v_1, \dots v_\ell\}$ of the associated dual complex $\Delta_\cX$.
    Moreover let $\sum_{i=1}^\ell b_i E_i = \cX_0$ be the irreducible decomposition.
    By Theorem~\ref{thm:continuity}, we can find a continuous functions $\psi\in \CPSH_A$ satisfying:
    \[\MA_A(\psi) = \mu.\]

    Now, consider the vertical $\R$-divisor on $\cX$ given by:
    \[G\doteq \sum_{i=1}^\ell \psi(v_{i})\,b_i\, E_i,\]
    and denote by $f$ the associated PL function $\varphi_G.$
    We will now prove that $\psi = P_A(f)$.
    
    For that, observe that for each $i\in \{1, \dots, \ell\}$ we have $\Pe(f)(v_i)\le f(v_i) = \psi(v_i)$. Therefore $P_A(f)\le \psi$ on the support of $\mu = \MA_A(\psi)$, thus by Corollary~\ref{cor:comparison} we obtain $P_A(f)\le \psi$ everywhere.
    On the other hand, applying Lemma~\ref{lem:P=sup} we have that 
    \begin{equation}\label{eq:P=sup}
        P_A(f) = \sup\{\varphi\in \PSH_A\mid \varphi(v_i)\le f(v_i)\}.
    \end{equation}
    The result then follows by observing that $\psi$ is a candidate on the supremum of the RHS of \eqref{eq:P=sup}.
\end{proof}

\section{Applications to the cscK problem}

The goal of this Section is to prove Theorem \ref{thm:A}, i.e. that uniform K-stability for models implies the existence of a unique cscK metric.

We will do this by proving that uniform K-stability for models is equivalent to uniform $\widehat{K}$-stability, see Proposition~\ref{prop:6.3}. Since we already know from \cite[Theorem A]{MP24transcendental} that uniform $\widehat{K}$-stability implies the existence of a unique cscK metric, this will then imply Theorem \ref{thm:A} .

We begin by giving a non-Archimedean interpretation of uniform K-stability for models.

\addtocontents{toc}{\SkipTocEntry}
\subsection{Intersection formulas for $\M_A$ and $\Jd_A$}

Recall from the introduction that for a big test configuration $(\cX,A+D)$ we define the Donaldson-Futaki invariant as
$$\DF(\mathcal{X},A+D):=K_{\mathcal{X}/\mathbb{P}^1}\cdot\langle(A+D)^n\rangle -\frac{n\alpha^{n-1}\cdot K_X}{(n+1)\alpha^n}\langle(A+D)^{n+1}\rangle,$$ the Mabuchi invariant as $$\M^{\NA}(\cX, A+D):=\DF(\mathcal{X},A+D)-(\cX_0-\cX_0^{red})\cdot \langle(A+D)^n\rangle$$ and the $J$-invariant as $$\Jd^{\NA}(\mathcal{X},A+D):=\langle A+D\rangle\cdot A^n-\frac{1}{n+1}\langle(A+D)^{n+1}\rangle.$$

Recall also from Section \ref{sec:M_A} that the non-Archimedean Mabuchi functional $\M_A$ is defined on $\mathcal{E}^1_A$ by
\[\M_A(\varphi):= \underline{s}\E_A(\varphi) + \E_A^{K_X}(\varphi) + \He_A(\varphi),\] where $\underline{s}:=-V^{-1}\alpha^{n-1}\cdot K_X$, and that the $J$-functional is defined by $$\Jd_A(\varphi):= -\E_A(\varphi)+\int\varphi \,\MA_A(0)=-\E_A(\varphi)+\sup \varphi.$$

The goal of this section is to prove the following result:
\begin{theorem} \label{thm:functionalid}
If $f_D\in PL_{\ge 1}$, then $$\M_A(P_A(f_D))=V^{-1}\M^{\NA}(\cX,A+D)$$ and $$\Jd_A(P_A(f_D))=V^{-1}\Jd^{\NA}(\mathcal{X},A+D).$$
\end{theorem}

For its proof we will need the following:

\begin{prop}\label{prop:restrictedintersection}
For all $i$ we have that
$$\langle (A+D)^n\rangle_{\cX|E_i}=\langle (A+D)^n\rangle\cdot E_i.$$
\end{prop}

\begin{proof}
It follows from the definitions that $$\langle (A+D)^n\rangle_{\cX|E_i}\le \langle (A+D)^n\rangle\cdot E_i.$$ We also have that $$\sum_i b_i\langle (A+D)^n\rangle\cdot E_i=\langle (A+D)^n\rangle\cdot \cX_1\le \langle (A+D)_{|\cX_1}^n\rangle=A^n\cdot \cX_1=V.$$ Together with Proposition \ref{prop:sumvol} this implies that $$\sum_i b_i\langle (A+D)^n\rangle_{\cX|E_i}=\sum_i b_i\langle (A+D)^n\rangle\cdot E_i$$ and hence $$\langle (A+D)^n\rangle_{\cX|E_i}=\langle (A+D)^n\rangle\cdot E_i$$ for all $i$.
\end{proof}

The algebraic version of the next result can be found in \cite[Proposition 3.2]{Li23fujita}. 
\begin{prop}\label{prop:entropyid}
If $f_D\in \PL_{\ge 1}$, then we have that
    \[\He_A(P_A(f)) = V^{-1} \langle (A+D)^n\rangle\cdot K^{\log}_{\cX/X\times \Pro}.\]

\end{prop}
\begin{proof}
Our proof follows that in \cite{Li23fujita}.
We recall that the entropy of $\varphi\in \cE^1_A$ is given by the formula
\[\He_A(\varphi) = \int (A_{X\times \Pro} -1)\,\MA_A(\varphi).\]
Thus by Proposition~\ref{prop:restrictedintersection} and Proposition~\ref{prop:MAform} we ge that $$\He_A(P_A(f_D))= V^{-1}\langle (A+D)^n\rangle\cdot\sum_i  (b_i\,E_i) (b_i^{-1} A_{X\times \Pro}(E_i) - 1).$$

Now we simply observe that by definition 
\begin{eqnarray*}
\sum_i  (b_i\,E_i) (b_i^{-1} A_{X\times \Pro}(E_i) - 1)=\sum_i  E_i (- b_i+ A_{X\times \Pro}(E_i) )=\\=-\cX_0 +\sum_i  E_i A_{X\times \Pro}(E_i)=- \cX_0 +K_{\cX/X\times \Pro} + \cX_0^{red}=K^{\log}_{\cX/X\times \Pro}.
\end{eqnarray*}
\end{proof}

To e.g. relate $\E_A(P_A(f_D))$ with $\langle(A+D)^{n+1}\rangle$ we will need the following fact:
\begin{lemma}
\label{lem:nef}
If $D\in \VCar(\cX)$ is effective and $A+D$ is relatively nef on $\cX$, then $A+D$ is actually nef on $\cX$.
\end{lemma}

\begin{proof}
Let $\gamma$ be a Kähler class on $\cX$ and $Z\subseteq \cX$ a $d$-dimensional subvariety. We then need to check that \[\left(A+D \right) \cdot \gamma^{d-1}\cdot [Z]\ge 0.\]
If $Z\subseteq \cX_t$ for some $t\in \Pro$, then by the relative nefness of $A+D$ and the numerical criterion applied on the fiber $\cX_t$, we have $(A+D)\cdot \gamma^{d-1}\cdot [Z]\ge 0$.
Now we consider the case when $Z$ is not contained in any fiber. By the numerical criterion $A$ has non-negative intersection with $\gamma^{d-1}\cdot [Z]$.
Hence, it is enough to check that 
    \[[D] \cdot [\gamma]^{d-1}\cdot [D]\ge 0.\]
Since $Z$ is not contained in any fiber, in particular it is no contained in the support of $D$, therefore the intersection $D\cdot Z$ is an effective $d-1$-cycle, which implies that 
    \[\left[D\right] \cdot [Z] \cdot [\gamma]^{d-1}\ge 0,\]
and we are done.
\end{proof}

\begin{prop} \label{prop:enpairposint}
    Let $f_D\in \PL_{\ge 1}$, $f_{D'}\in PL$ and $\Gamma = \pi_X^* \gamma$ where $\gamma \in H^{1,1}(X).$ For any $\ell$ we then have that
\begin{equation} \label{eq:genint}
(A, P_A(f_D))^\ell\cdot  (\Gamma, f_{D'})^{n+1-\ell}=\langle (A+D)^\ell\rangle \cdot (\Gamma+D^\prime)^{n+1-\ell},
\end{equation}
    where the latter intersection is taken on a SNC model $\cX$ representing both $D$ and $D^\prime$. 

In particular we get that
\begin{equation} \label{eq:energyvolume}
        \E_A(P_A(f_D)) = \frac{V^{-1}}{n+1}\langle (A+D)^{n+1}\rangle,
\end{equation}
\begin{equation} \label{eq:twistedenergy}
\E_A^{K_X}(P_A(f_D)) =V^{-1} \langle (A+D)^n\rangle\cdot \pi_X^*K_X,
\end{equation}
and
\begin{equation} \label{eq:Jfunc}
\Jd_A(\varphi) = \langle A+D\rangle\cdot A^n-\frac{1}{n+1}\langle(A+D)^{n+1}\rangle.
\end{equation}
\end{prop}
\begin{proof}
We start by proving (\ref{eq:genint}). 

Let $D_k$ be a sequence of effective divisors such that $A_k:=A+D-D_k$ are relatively K\"ahler on some models $\cX_k$, and such that $0\le \varphi_k = f_{D-D_k}$ converges to $P_A(f)$.

Let us show that $A_k^\ell\cdot (\Gamma+D^\prime)^{n+1-\ell}\to \langle (A+D)^\ell\rangle\cdot(\Gamma+D^\prime)^{n+1-\ell}.$
To do so we prove that $A_k^\ell\cdot B\to\langle (A+D)^{\ell}\rangle\cdot B$, for any semipositive class $B\in H^{n+1-l,n+1-l}(\cX)$. 
    
By definition for any fixed $\epsilon>0 $ sufficiently small, we can find an effective divisor $G$ such that $A^\prime:=A+D-G$ is K\"ahler on some model $\cX^\prime$, and such that 
    \[\langle A+D\rangle^\ell \cdot B - \epsilon \le (A^\prime)^\ell \cdot B \le \langle A+D\rangle^\ell \cdot B.\]
    Since $E_{nK}(A+D)\subseteq \cX_0$ we can assume that the $D^\prime$ is vertical. 
    If we now let $\psi:=\varphi_{D-G}\in \cH_A$ then for any $\delta>0$ there exists a $k\in \mathds N$ such that 
    $\psi - \delta \le \varphi_k.$

    Now, by Lemma~\ref{lem:nef} we have that $A_k$ and $A^\prime - \delta \cX_0^\prime$ are nef for $\delta$ small enough.
    Therefore, since the difference $A_k- A^\prime -\delta \cX_0^\prime = D_k - G - \delta \cX_0^\prime$ is effective we have
    \[(A^\prime)^\ell\cdot B -O(\delta)= (A^\prime - \delta \cX_0^\prime)^\ell\cdot B \le A_k^\ell\cdot B.\]

    For $\delta$ sufficiently small we therefore have 
    \[(A^\prime)^\ell \cdot B - \epsilon \le A_k^\ell\cdot B\le \langle A+D\rangle ^\ell \cdot B,\]
    which implies that 
    \[\langle A+D\rangle^\ell \cdot B - \epsilon \le A_k^\ell\cdot B\le \langle A+D\rangle ^\ell \cdot B,\]
    as we wanted.

    Moreover, since $\varphi_k$ is continuous and increases to $P_A(f_D)$, which is continuous, by Dini's Lemma it converges uniformly to $P_A(f_D)$. Lemma \ref{lem:enpairest} then applies and we have:
    \begin{align*}
        (A, P_A(f_D))^\ell \cdot (\Gamma, f_{D^\prime})^{n+1-\ell}&= \lim_{k} (A, \varphi_k)^\ell \cdot (\Gamma, f_{D^\prime})^{n+1- \ell}\\
        &=\lim_k A_k^\ell \cdot (\Gamma + D^\prime)^{n+1-\ell} = \langle A+D\rangle^\ell \cdot (\Gamma + D^\prime)^{n+1-\ell},
    \end{align*}
    which completes the proof of (\ref{eq:genint}).

We now observe that the formulas (\ref{eq:energyvolume}), (\ref{eq:twistedenergy}) and (\ref{eq:Jfunc}) all follow directly from (\ref{eq:genint}), hence we are done.
    
\end{proof}

Now we are ready to prove Theorem \ref{thm:functionalid}.

\begin{proof}[Proof of Theorem~\ref{thm:functionalid}]
By Propositions \ref{prop:entropyid} and \ref{prop:enpairposint} we get that 
\[V\M_A(P_A(f_D)) =  \frac{\underline s}{n+1}\langle (A+D)^{n+1}\rangle +\langle (A+D)^{n} \rangle \cdot \pi_X^* K_X  + \langle (A+D)^n\rangle\cdot K^{\log}_{\cX/X\times \Pro}. \]

Since \[K_{X\times\Pro} = \pi_X^*K_X + \pi_{\Pro}^* K_{\Pro}\] we have $ K^{\log}_{\cX/\Pro} = K^{\log}_{\cX/X\times\Pro} + \pi^*_XK_X$, and thus:
\[V\M_A(P_A(f)) =\frac{\underline s}{n+1}\langle (A+D)^{n+1}\rangle +\langle (A+D)^{n} \rangle \cdot K^{\log}_{\cX/\Pro}, \]
which implies that
\[\M_A(P_A(f)) = V^{-1}\M^{\NA}(A+D).\]
\end{proof}

We immediately get:
\begin{corollary}
    Uniform K-stability for models, as defined in the introduction, is equivalent to there being a $\delta>0$ such that $$\M_A(P_A(f)) \ge \delta \Jd_A(P_A(f))$$ for all $f\in PL$.
    In particular, uniform $\widehat K$-stability implies uniform K-stability for models.
\end{corollary}

\addtocontents{toc}{\SkipTocEntry}
\subsection{K-stability for models and $\widehat K$-stability}

We will now use our non-Archimedean Calabi-Yau Theorem to prove that uniform K-stability for models actually is equivalent to uniform $\widehat K$-stability. The argument follows \cite[Proposition 6.3]{Li22geodesic}.

\begin{theorem}
    \label{prop:6.3}
    The following are equivalent:
    \begin{enumerate}
        \item There exits $\delta>0$ such that \[\mathrm M_A(\varphi)\ge \delta \mathrm J_A(\varphi) \text{ for every } \varphi \in \cE^1_A.\]
        \item There exits $\delta>0$ such that \[\mathrm M_A(\varphi)\ge \delta \mathrm J_A(\varphi) \text{ for every } \varphi \in \CPSH_A,\]
        whose Monge--Ampère measure is supported on some dual complex $\Delta_\cX$.
        \item There exists $\delta>0$ such that for every $f\in \PL$ we have \[\M_A\left( P_A(f)\right)\ge \delta\mathrm J_A\left( P_A(f)\right).\]
    \end{enumerate}
    In particular, uniform $\widehat K$-stability is equivalent to uniform K-stability for models.
\end{theorem}
\begin{proof}
    It is clear that, by the continuity of envelopes and Proposition~\ref{prop:MAform}, (1) implies (2) which implies (3).

    Let us show that (2) implies (1): we will prove that for every $\varphi \in \cE_A^1$ there exists a sequence $\varphi_j\in \CPSH_A$ such that 
    \[\varphi_j\overset{s}{\longrightarrow}\varphi, \text{ and also } \mathrm M_A(\varphi_j) \to \mathrm M_A(\varphi).\]
    Also, since both $\Jd_A$ and $\M_A$ are translation invariant we can suppose that $\sup \varphi = 0 = \sup \varphi_i.$

    Let $\mu$ be the Monge–Ampère measure of $\varphi$.
    By Corollary~\ref{cor:energyentropy} we can find a sequence of measures $\mu_j$ supported on dual complexes $\Delta_{\cX^j}$ converging strongly to $\mu$ such that
    \begin{equation}\label{eq:convergences}
        \Ent(\mu_j)\to \Ent(\mu).
    \end{equation}

    Applying Theorem~\ref{thm:continuity}, we obtain a sequence of continuous functions $\varphi_j$ satisfying 
    \[\MA_A(\varphi_j) = \mu_j,\]
    which by Theorem~\ref{thm:B} implies that $\varphi_j$ converges strongly to $\varphi$. In particular, we get $\Jd_A(\varphi_j)\to \Jd_A(\varphi)$.
    Furthermore, by the entropy convergence of \eqref{eq:convergences}, we  get:
    \[\M_A(\varphi_j)\to \M_A(\varphi).\]
    We will now show that (3) implies (2). Thus we will prove that for every $\varphi\in \CPSH_A$, whose Monge–Ampère measure $\mu:=\MA_A(\varphi)$ is supported on a dual complex $\Delta_\cX$, we can find a sequence $f_j\in \PL$ such that $\M_A(P_A(f_j))\to \M_A(\varphi)$ and $\Jd_A(P_A(f_j))\to \Jd_A(\varphi)$.

    To do so, we start by observing that  Corollary~\ref{cor:continuousenevlope} gives that if $\varphi\in \CPSH_A$ is as above, then $\varphi  = P_A(\varphi\circ p_\cX).$ 

    Now, since $\varphi\circ p_\cX$ is a continuous function invariant by $p_\cX$, we can find a sequence of PL functions $f_j\in \PL$ such that:
    \begin{itemize}
        \item The sequence converges $f_j$ uniformly to $\varphi\circ p_\cX$. 
        \item We can find a model $\cX^j$ together with a vertical divisor $D_j\in \VCar(\cX^j)$ such that $f_j = f_{D_j}$, whose dual complex $\Delta_{\cX^j}$ is just a subdivision of $\DcX$.
        That is, $\cX^j$ can be obtained by a sequence of blow-ups of smooth centers of $\cX$.
    \end{itemize}
    As a consequence of the first point we have that $P_A(f_j)$ converges uniformly and hence strongly to $P_A(\varphi\circ p_\cX) = \varphi$. This implies that $\Jd_A(P_A(f_j))$ converges to $\Jd_A(\varphi).$
    
    Moreover, by Proposition~\ref{prop:MAform} we have that $\mu_j:=\MA(P_A(f_j))$ is supported on the set $\Delta_{\cX^j}\hookrightarrow X^{\NA}$.
    Since the dual complex $\Delta_{\cX^j}$ is a subdivision of $\DcX$, as valuations they coincide \[i_{\cX^j}(\Delta_{\cX^j}) = i_{\cX}(\Delta_\cX) \subseteq X^{\NA}.\]
    Therefore $\mu_j$ is supported on $\DcX\hookrightarrow X^{\NA}$ as well.
    By Lemma~\ref{lem:weakentropy} we have 
    \[\Ent(\mu_j) \to \Ent(\mu),\]
    that, together with the strong convergence $\mu_j\overset{s}{\rightarrow}\mu$, gives the convergence $\M_A(P_A(f_j))\to \M_A(\varphi).$
\end{proof}

This finally allows us to prove Theorem~\ref{thm:A}.
\begin{corollary}[Theorem \ref{thm:A}]
    If $(X,\alpha)$ is uniformly K-stable over models, then there is a unique cscK metric in $\alpha$.
\end{corollary}
\begin{proof}
    If $(X, \alpha)$ is uniformly K-stable over models, then by Theorem~\ref{prop:6.3}, $(X, \alpha)$ is uniformly $\widehat K$-stable, and the existence of a unique cscK metric in $\alpha$ thus follows from \cite[Theorem A]{MP24transcendental}.
\end{proof}

\section{A valuative criterion for K-stability for models}\label{sec:valuative}

\subsubsection{Introduction to valuative criteria}

When $X$ is a Fano manifold, and $\alpha = \chern_1(X)$, Fujita and Li \cite{Fuj19valuative, Li17Ksemi}, relying on the work of Li and Xu \cite{LX14special}, gave a \emph{valuative criterion for K-stability}.
They introduced a numerical invariant $\beta\colon \Xdiv\to \R$, which for each divisorial valuation $v = \ord_F$ associates a real number given by:
\[\beta(v) = \beta(F):= A_X(v) - \int_0^{+\infty}\langle(-K_X + \lambda F)^n\rangle \mathrm d\lambda.\]

They then proved that the positivity of this invariant in a special class of divisorial valuations is equivalent to K-stability of $X$.
More concretely if $\beta(v)>0$, for all the divisorial valuations $v$ that arise from a \emph{dreamy divisor} –a divisor that induces an ample (special) test configuration\footnote{This can be formulated in the non-Archimedean dictionary by saying that there exists $\varphi\in \cH_A$ such that $\MA_A(\varphi) = \delta_{v}$}– then $(X, -K_X) $ is K-stable and conversely if $X$ is K-stable then $\beta(v)>0$ for all such $v$.

Dervan and Legendre in \cite{DL23valuative} extended this invariant to a general polarized $(X, L)$ and proved that K-stability with respect to test configurations with integral central fiber is equivalent to the positivity of the $\beta$-invariant in this more general setting.

Furthermore, Boucksom and Jonsson in \cite{BJ23nakstabii} developed a valuative criterion for K-stability for models in the sense of Chi Li \cite{Li22geodesic} for a polarized variety.
This criterion, however, is different from the one in \cite{DL23valuative, Fuj19valuative, Li17Ksemi}.
Instead of dealing with the positivity of the $\beta$-invariant of single prime (possibly dreamy) divisors, it deals with positivity of some invariant (that we will also call the $\beta$-invariant) on the set of divisorial measures, that is the data of weighted combination of divisorial valuations.
The objective of this section is to generalize the latter for Kähler manifolds.
In Section~\ref{sec:computingbeta} we also give a formula for the measure theoretic $\beta$-invariant that --up to taking derivatives and Legendre transforms-- involves computing only log discrepancies and integrals of volumes of the form 
\[\langle(\mu^*\alpha + \sum s_i\, F_i)^n\rangle, \quad \text{ for } s_i\ge 0, \]
as in the \cite{Fuj19valuative, Li17Ksemi} $\beta$-invariant.

\subsubsection{The $\beta$-invariant}

Following \cite{BJ21nakstabi, BJ23nakstabii}, to each measure of finite energy $\mu \in \cM^1$, we define:
\begin{align*}
    \beta_A(\mu)&:= \Ent(\mu) +\nabla_{K_X} \Ev_A(\mu)\\
    &=\Ent(\mu) + \frac{\mathrm d}{\mathrm d t}\bigg| _{t=0} \Ev_{A + t K_X}(\mu).
\end{align*}
We observe that the above expression is well defined since the differentiability of the energy functional was established in \cite{BJ23synthetic}.

Moreover, by \cite[Equation 4.2 and Proposition 4.5]{BJ23synthetic} we also find that for every $\varphi\in \cH_A$ the derivative of the energy of the Monge–Ampère measure
\begin{equation}\label{eq:beta}
    \dtz \Ev_{A + t K_X}(\MA_A(\varphi)) = \underline s \E_A(\varphi) + \E_A^{K_X}(\varphi), 
\end{equation}
is equal to the Mabuchi energy of $\varphi$. More generally we have
\begin{theorem}\label{thm:beta}
    Let $\varphi \in \cE^1_A$ be a finite energy potential, we then have that 
    \[\beta_A(\MA_A(\varphi)) = \M_A(\varphi), \quad \text{and} \quad \Ev_A(\MA_A(\varphi)) \approx \Jd_A(\varphi).\]
\end{theorem}
\begin{proof}
    For the beta invariant, when $\varphi\in \cH_A$ this follows directly from Equation~\eqref{eq:beta}.
    For $\varphi\in \cE^1_A$, it is enough to consider a decreasing sequence $\varphi_i\in \cH_A$ converging to $\varphi$, and apply Proposition 4.5 (ii) of \cite{BJ23synthetic} for the Monge--Ampère measures of $\varphi$ and $\varphi_i$ to obtain that
    \[\nabla_{K_X} \Ev_A(\MA_A(\varphi_i))\to \nabla_{K_X} \Ev_A(\MA_A(\varphi)).\]
    Since $\underline s \E_A(\varphi_i) + \E_A^{K_X}(\varphi_i)$ also converges to $\underline s \E_A(\varphi) + \E_A^{K_X}(\varphi)$ the result follows.

    As for the energy, it is enough to observe that 
    \begin{align*}
        \Ev_A(\MA_A(\varphi)) &= \E_A (\varphi) - \int \varphi \,\MA_A(\varphi) \\
        &= \Jd_A(\varphi, 0)\approx \Jd_A(0, \varphi) = \Jd_A(\varphi).
    \end{align*}
\end{proof}

As a trivial consequence we have
\begin{corollary}\label{cor:measures}
    Uniform $\widehat K$-stability is equivalent to
    \[\inf_{\mu\in \cM^1}\frac{\beta(\mu)}{\Ev(\mu)}>0.\]
\end{corollary}
In turn, this implies that uniform $\widehat K$-stability is an open condition:
\begin{corollary}
    Uniform $\widehat{K}$-stability is an open condition in the Kähler class $\alpha\in H^{1,1}(X)$.
\end{corollary}
\begin{proof}
    This follows from Corollary~\ref{cor:measures} together with \cite[Theorem 5.5]{BJ23synthetic}.
\end{proof}

Finally, we state the following valuative criterion for uniform $\widehat K$-stability:
\begin{theorem}[Valuative criterion for $\widehat K$-stability]
    Uniform $\widehat K$-stability is equivalent to 
    \[\inf_{\mu\in \Md}\frac{\beta_A(\mu)}{\Ev_A(\mu)}>0.\]
\end{theorem}
\begin{proof}
    This follows from Propositions \ref{prop:divisorialenvelope} and \ref{prop:6.3} together with Theorem~\ref{thm:beta}.
\end{proof}

\addtocontents{toc}{\SkipTocEntry}
\subsection{Computing the beta invariant}
\label{sec:computingbeta}
The main goal of this section is to give a description of how to compute the $\beta$-invariant of a divisorial measure, the key object to study the above mentioned valuative criterion for $\widehat K$-stability. 
We do so in spirit of the valuative criterions in the Fano setting \cite{Li17Ksemi, Fuj19valuative}. 
In particular, we will use a different description of divisorial valuations as the one taken in this paper so far.
If we let $F\subseteq Y\overset{\mu}{\longrightarrow} X$ be a prime divisor on some modification of $X$, we can define
\[\ord_F\colon \ideal_X\to \R, \quad I\mapsto \inf\{\ord_F(g\circ\mu) : g\in I_x\},\]
for any generic choice of $x\in \mu(F).$
One of the main results of \cite{MP24transcendental} is to show that $\ord_F\in \Xdiv$ is a divisorial valuation, and moreover that, up to rescaling by a rational number, all divisorial valuations are of this form. 
Furthermore, as in \cite{MP24transcendental}, one can see that if $v = \ord_F$ is a divisorial valaution, then the log-discrepancy satisfies $A_{X\times\Pro}(v) = A_X(F)+1$.

Let $\{F_1, \dotsc, F_\ell\}$ be a finite set of prime divisors over $X$, $\{v_1, \dotsc, v_\ell\}$ the associated divisorial valuations and let $\mu_\xi$ be the measure \[\mu_\xi:=\sum_{i=1}^\ell \xi_i \, \delta_{v_i},\]
for $(\xi_1, \dotsc, \xi_\ell)\in \{x\in (\Q_{\ge 0})^\ell\mid \sum x_i =1\},$ we will denote $\mu_\xi$ by simply $\mu$ in the following.
We recall that
\begin{align*}
    \beta_A(\mu) &= \Ent(\mu) + \nabla_{K_X}\Ev_A(\mu)\\
    &= \sum_{i=1}^\ell \xi_i\cdot A_X(F_i) + \nabla_{K_X}\Ev_A(\mu).
\end{align*}
Hence, to calculate the $\beta$-invariant of $\mu$ one needs to compute the log discrepancy of the divisors $F_1, \dotsc, F_\ell$ and also to compute the energy $\Ev_\zeta(\mu)$ explicitly for some Kähler classes $\zeta\in H^{1,1}(X)$. More precisely it enough to compute $\Ev_{\alpha + tK_X}(\mu)$, for $t$ sufficiently small.

Recall that the energy $\Ev_A(\mu)$ is defined as 
\[\sup\left\{ \E_A(\varphi) - \int \varphi\, \mathrm d\mu \mid \varphi \in \cE^1_A\right\}.\]

If $t= (t_1, \dotsc, t_\ell)\in \R^\ell$, we denote by $\varphi_t$ the potential 
\begin{equation}\label{eq:phit}
    \varphi_t:=\{\varphi\in \PSH(\alpha)\mid \varphi(v_i)\le t_i \text{ for every } i= 1, \dotsc, \ell\} \in \CPSH_A.
\end{equation}
By Proposition~\ref{prop:divisorialenvelope} and Equation~\eqref{eq:P=sup} there exists $t^\star \in\R^\ell$ such that the associated potential $\varphi_{t^\star}$ maximizes the energy 
\begin{align*}
    \Ev_A(\mu) &=\E_A(\varphi_{t^\star}) - \int \varphi_{t^\star}\, \mathrm d\mu
\end{align*}
which, as we saw before, is equivalent to solving the equation
\(\MA_A(\varphi_{t^\star}) = \mu.\)
Moreover, by orthogonality we get
\[\Ev_A(\mu)= \E_A(\varphi_{t^\star}) - \sum \xi_i\cdot t^\star_i.\]
In particular, 
\begin{equation}\label{eq:energyopt}
\begin{aligned}
    \Ev_A(\mu_\xi) &\ge \sup \left\{ \E_A(\varphi_t) - \int \varphi_t\, \mathrm d\mu_\xi  \mid t\in \R^\ell\right \}\\
    &\ge \sup \left\{ \E_A(\varphi_t) - \xi\cdot t \mid t\in \R^\ell\right \}\\
    &\ge \E_A(\varphi_{t^\star}) - \xi\cdot t^\star = \Ev_A(\mu_\xi).
\end{aligned}
\end{equation}

In this way, we obtain the following finite dimensional optimization description of the energy $\Ev_A(\mu)$: 
denoting by $f_A\colon \R^\ell \to \R$ the function \[t\mapsto\E_A(\varphi_t),\] 
we get
\[\Ev_A(\mu_\xi)=\sup_{t\in \R^\ell} \left \{f_A(t) - \xi\cdot t\right\}.\]
What we will do next is to compute $f_A$ in terms of the expressions of the form 
\begin{equation}\label{eq:volumeexpression}
    \langle(\alpha - \sum_{i=1}^\ell s_i F_i)^n \rangle, \text{ for } s\in (\R_{\ge 0})^\ell.
\end{equation}
Using Equation~\eqref{eq:energyopt} we see that the Legendre transform of $-f$ at the point $-\xi$ computes the energy of $\mu_\xi$
\[\Ev_A(\mu_\xi) = \widehat{-f_A}\,(-\xi).\]

\subsubsection{Computing $f_A(t)$}
In order to compute the function $f_A(t)$ we will do as follows:
\begin{enumerate}
    \item Use Theorem 5.3.4 of \cite{MP24transcendental} to get a maximal geodesic ray $U$ associated to $\varphi_t$ such that
        \[U^{\NA} = \varphi_t, \quad \text{and}\quad \lim_{s\to \infty} \frac{\E_\omega(U_s)}{s} = \E_A(\varphi_t) = f(t). \]
    \item Give an explicit description of the Ross–Witt Nyström transform of $U$, so we can apply \cite[Theorem 2.6 (iii)]{DXZ23transcendental} to compute the energy of $U$ in terms of expressions of the form of \eqref{eq:volumeexpression}.
\end{enumerate}

Let us fix $t\in \R^\ell$, and denote by $\cF_\lambda$ the set 
\[\{u\in \PSH_{\sup}(\omega)\mid \nu(u, F_i)\ge \lambda - t_i \text{ for every } i = 1, \dotsc, \ell\},\]
where $\PSH_{\sup}$ denotes the set of $\omega$-psh functions with supremum normalized to be zero, and $G_\lambda$ denote the $\omega$-psh function
\(\sup\{u\in \cF_\lambda\}.\)

It is clear that $(G_\lambda)_\lambda$ is a \emph{relative test curve}, that is for every $x\in X$ the map  $\lambda\mapsto G_\lambda(x)$ is
\begin{itemize}
    \item Decreasing: if $\lambda \le \mu$ then $\cF_\mu \subseteq \cF_\lambda$, and thus $G_\mu\le G_\lambda$.
    \item Concave: for any $t\in [0,1]$ and any two functions $u_1\in \cF_\lambda$, $u_2\in \cF_\mu$, we have $t\cdot u_1 + (1-t)\cdot u_2\in \cF_{t\lambda + (1-t)\mu},$ implying $t\cdot G_\lambda + (1-t)\cdot G_\mu \le G_{t\lambda + (1-t)\mu}.$
    \item Upper semicontinuous: This follows from the fact that $\bigcap_{\epsilon>0}\cF_{\lambda + \epsilon} = \cF_\lambda.$
\end{itemize}

It is also clear that by definition $G_\lambda$ is \emph{$\cI$-maximal}, that is 
\[\mathrm P[G_\lambda]_{\cI} = G_\lambda.\]
We also observe that since Lelong numbers are bounded, there exists $\lambda_{0}\in \R$ such that for every $\lambda\ge \lambda_{0}$ the set $\cF_\lambda$ is empty, we denote $\lambda_{\max}$ the infimum of all such $\lambda_0$.
Let us also denote by $\lambda_{\min}(t)$, or simply $\lambda_{\min}$, the quantity $\min\{t_1, \dotsc, t_\ell\}$, and point out that for every $\lambda\le \lambda_{\min}$ the function $G_\lambda $ is constant equal to zero.

\begin{lemma}
    The Ross–Witt Nyström transform of $G_\lambda$ is the maximal geodesic ray starting from $0$ associated to $\varphi_t$.
\end{lemma}
\begin{proof}
Let $U$ be a maximal geodesic ray associated to $\varphi_t$.  As in \cite[Proposition 3.1]{DXZ23transcendental} ––whose proof, originally for the projective setting, applies as is to our transcendental setting–– we have that for every divisorial valuation $v = r\cdot\ord_F$
\[\sup\{ -r\,\nu(\check U_\lambda, F) + \lambda \mid \lambda\le \lambda_{\max}\} = U^{\NA}(v) = \varphi_t(v),\]
if we normalize so that $U_0 = 0$, we observe
\[0=U_0= \sup_{\lambda\in \R} \check U_\lambda + \lambda \cdot 0 = \sup_{\lambda\in \R} \check U_\lambda,\]
thus for every $\lambda \in \R$ the supremum $\sup \check U_\lambda \le 0$, moreover we have for every $i = 1, \dotsc, \ell$ and every $\lambda\le \lambda_{\max}$
\[- \nu(\check U_\lambda, F_i) + \lambda \le \varphi_t(\ord_{F_i})\le t_i, \iff \nu(\check U_\lambda , F_i)\ge \lambda - t_i\]
which in turn implies $\check U_\lambda \le G_\lambda.$
Since $G_\lambda$ is $\cI$-maximal (in particular maximal), the Ross–Witt Nyström transform of $G_\lambda$ is a psh geodesic ray, which satisfies
\[U_s \le \hat G_s, \quad \text{and for every }i = 1, \dotsc, \ell \quad \hat G^{\NA}(\ord_{F_i})\le t_i\]
where the second inequality is again obtained using \cite[Proposition 3.1]{DXZ23transcendental}.
Combining the previous inequalities we have
\[\varphi_t = U^{\NA} \le \hat G^{\NA} \le \varphi_t,\]
which, by maximality of $U$, implies that $U = \hat G$.
\end{proof}

As a consequence we compute $f_A(t)$:

\begin{theorem}\label{thm:f}
    For every $t\in \R^\ell$ the energy of $\varphi_t$
    \[f_A(t) = \lambda_{0}+V^{-1}\int_{\lambda_{0}}^{+\infty}\langle(\alpha -\sum_{i=1}^\ell(\lambda - t_i)_+\, F_i)^n\rangle\mathrm d \lambda,\]
    where $\lambda_{0}$ is any constant strictly less than $\min\{t_1, \dotsc, t_\ell\}.$
\end{theorem}
In the algebraic case the same formula appears without proof in \cite{Li23notes}.

\begin{proof}
Once again, let $U$ be the maximal geodesic ray associated to $\varphi_t$ starting from $0$, and recall that by the previous Lemma, its Ross–Witt Nyström transform is given by $\lambda\mapsto G_\lambda$.

By \cite[Theorem 2.6 (iii)]{DXZ23transcendental}  
\begin{equation}\label{eq:energy}
f_A(t) =\frac{\E_\omega(U_s)}{s} = \lambda_{\max} + \int_{-\infty}^{\lambda_{\max}} \left(-1 + V^{-1}\int_X \langle \omega +\ddc G_\lambda \rangle^n \right)\mathrm d\lambda. 
\end{equation}
 
Let us compute the last energy term.
Let $\theta_{F_i}$ be a smooth form representing the first Chern class of the line bundle associated to $F_i$, $\chern_1(\cO(F_i))$.
Let $\psi_{F_i}$ then be a potential satisfying
\[\theta_{F_i} + \ddc \psi_{F_i} = \delta_{F_i},\]
where $\delta_{F_i}$ represents the current of integration along $F_i$.
For simplicity, denote by $\theta^\lambda_F$ the smooth form $\sum_{i=1}^\ell (\lambda - t_i)_+\, \theta_{F_i},$ by $\psi^\lambda_F$ the function $\sum_{i=1}^\ell (\lambda - t_i)_+\, \psi_{F_i}$, and by $\delta^\lambda_F$ the measure $\sum_{i=1}^\ell (\lambda - t_i)_+\, \delta_{F_i}$.
Then the current 
\begin{align*}
    \omega - \theta^\lambda_{F} + \ddc \left[G_\lambda -\psi^\lambda_F\right]= \omega +\ddc G_\lambda -\delta_F^\lambda\in \alpha -\sum_{i=1}^\ell(\lambda - t_i)_+\, F_i
\end{align*}
has minimal singularities in $\alpha -\sum_{i=1}^\ell(\lambda - t_i)_+\, F_i$.
Furthermore, the integral of the nonpluripolar product satisfies:
\[\int_X \langle(\omega +\ddc G_\lambda)^n\rangle = \int_X \langle( \omega +\ddc G_\lambda - \delta^\lambda_F)^n\rangle.\]
Since the latter has minimal singularities, we then have that
\[\int_X \langle (\omega +\ddc G_\lambda)^n\rangle =\langle(\alpha -\sum_{i=1}^\ell(\lambda - t_i)_+\, F_i)^n\rangle.\]
Adding up everything together we obtain the desired formula:
\begin{equation}
    \begin{aligned}
        f_A(t) = \E_A(\varphi_t) &= \lambda_{\max} + \int_{-\infty}^{\lambda_{\max}} \left(-1 + V^{-1}\int_X \langle (\omega +\ddc G_\lambda)^n \rangle \right)\mathrm d\lambda\\
        &= \lambda_{\max} + \int_{\lambda_{0}}^{\lambda_{\max}} \left(-1 + V^{-1}\int_X \langle (\omega +\ddc G_\lambda)^n \rangle \right)\mathrm d\lambda\\
        &= \lambda_{0} + \int_{\lambda_{0}}^{\lambda_{\max}} \left( V^{-1}\int_X \langle (\omega +\ddc G_\lambda)^n \rangle \right)\mathrm d\lambda\\
        &=\lambda_{0}+V^{-1}\int_{\lambda_{0}}^{\lambda_{\max}} \langle(\alpha -\sum_{i=1}^\ell(\lambda - t_i)_+\, F_i)^n\rangle\mathrm d \lambda.
    \end{aligned}
\end{equation}
where $\lambda_{0} < \min\{t_1, \dotsc, t_\ell\}$.
\end{proof}

Specializing the previous result for the case $\ell =1$, we get the known formula
\[\E_A(\varphi_v) = V^{-1}\int_0^{\tau_{\psef}} \langle(\alpha - \lambda F)^n\rangle \,\mathrm d\lambda.\]

\begin{corollary}
    The Legendre transform of $ -f_A(t)$ is given by
    \[\widehat{-f_A}\,(-\xi) = \sup_{t\in \R^\ell}\left\{- \langle\xi, t\rangle + f_A(t)\right\} = \Ev_{A}(\mu_\xi),\]
    where $\mu_\xi \doteq \sum_{i=1}^\ell \xi_i \delta_{v_i}$ is divisorial measure associated to $\xi$ and $F_1, \dotsc, F_\ell.$
\end{corollary}
If we denote by $g_A(\xi)$ the Legendre transform $\widehat{-f_A}\, (-\xi)$, for $f_A(t)$ as in Theorem~\ref{thm:f} we get the formula:
\begin{corollary} \label{cor:betaformula}
    Let $\mu_\xi$ be like above, then 
    \begin{align*}
        \beta_A(\mu_\xi) = \sum_{i=1}^\ell \xi_i \cdot A_X(F_i) + \nabla_{K_X} \, g_A(\xi). 
    \end{align*}
\end{corollary}
As a moral conclusion we have that in order to compute the beta invariant of a divisorial measure with support on $\ord_{F_1}, \dotsc, \ord_{F_\ell}$ it is enough to be able to compute the log discrepancy of $F_1, \dotsc, F_\ell$ and the integral of positive intersections of the form 
\[\langle(\alpha + \sum_{i=1}^\ell s_i F_i)^n\rangle, \quad \text{ for } s\in (\R_+)^\ell.\]

\bibliography{pietro}
\bibliographystyle{alpha}
\end{document}